\pdfoutput=1
\RequirePackage{ifpdf}
\ifpdf % We are running pdfTeX in pdf mode
\documentclass[pdftex]{sigma}
\else
\documentclass{sigma}
\fi

\numberwithin{equation}{section}

\newtheorem{Theorem}{Theorem}[section]
\newtheorem{Corollary}[Theorem]{Corollary}
\newtheorem{Lemma}[Theorem]{Lemma}
\newtheorem{Conjecture}[Theorem]{Conjecture}
\newtheorem{Proposition}[Theorem]{Proposition}
 { \theoremstyle{definition}
\newtheorem{Definition}[Theorem]{Definition}
\newtheorem{Remark}[Theorem]{Remark}}

\usepackage{pinlabel}
\usepackage[all]{xy}

\newcommand{\R}{\mathbb {R}}

\newcommand{\Z}{\mathbb {Z}}
\newcommand{\F}{\mathbb {F}}

\newcommand{\alg}{\mathcal{A}}

\newcommand{\wt}{\widetilde}

\newcommand{\aug}{{\rm Aug}}

\newcommand{\leg}{\mathfrak{Leg}}

\def\wt#1{\widetilde{#1}}

\begin{document}

\allowdisplaybreaks

\newcommand{\arXivNumber}{1908.08978}

\renewcommand{\thefootnote}{}

\renewcommand{\PaperNumber}{017}

\FirstPageHeading

\ShortArticleName{Legendrian DGA Representations and the Colored Kauffman Polynomial}

\ArticleName{Legendrian DGA Representations\\ and the Colored Kauffman Polynomial\footnote{This paper is a~contribution to the Special Issue on Algebra, Topology, and Dynamics in Interaction in honor of Dmitry Fuchs. The full collection is available at \href{https://www.emis.de/journals/SIGMA/Fuchs.html}{https://www.emis.de/journals/SIGMA/Fuchs.html}}}

\Author{Justin MURRAY~$^\dag$ and Dan RUTHERFORD~$^\ddag$}

\AuthorNameForHeading{J.~Murray and D.~Rutherford}

\Address{$^\dag$~Department of Mathematics, 303 Lockett Hall, Louisiana State University,\\
\hphantom{$^\dag$}~Baton Rouge, LA 70803-4918, USA}
\EmailD{\href{mailto:jmurr24@lsu.edu}{jmurr24@lsu.edu}}

\Address{$^\ddag$~Department of Mathematical Sciences, Ball State University,\\
\hphantom{$^\ddag$}~2000 W.~University Ave., Muncie, IN 47306, USA}
\EmailD{\href{mailto:rutherford@bsu.edu}{rutherford@bsu.edu}}

\ArticleDates{Received August 28, 2019, in final form March 10, 2020; Published online March 22, 2020}

\Abstract{For any Legendrian knot $K$ in standard contact $\R^3$ we relate counts of ungraded ($1$-graded) representations of the Legendrian contact homology DG-algebra $(\mathcal{A}(K),\partial)$ with the $n$-colored Kauffman polynomial. To do this, we introduce an ungraded $n$-colored ru\-ling polynomial, $R^1_{n,K}(q)$, as a linear combination of {\it reduced} ruling polynomials of positive permutation braids and show that (i) $R^1_{n,K}(q)$ arises as a specialization $F_{n,K}(a,q)\big|_{a^{-1}=0}$ of the $n$-colored Kauffman polynomial and (ii) when $q$ is a power of two $R^1_{n,K}(q)$ agrees with the total ungraded representation number, $\operatorname{Rep}_1\big(K, \mathbb{F}_q^n\big)$, which is a normalized count of $n$-dimensional representations of $(\mathcal{A}(K),\partial)$ over the finite field $\mathbb{F}_q$. This complements results from~[Leverson C., Rutherford D., \textit{Quantum Topol.} \textbf{11} (2020), 55--118] concerning the colored HOMFLY-PT polynomial, $m$-graded representation numbers, and $m$-graded ruling polynomials with $m \neq 1$.}


\Classification{53D42; 57M27}

\renewcommand{\thefootnote}{\arabic{footnote}}
\setcounter{footnote}{0}

\rightline{\it To Dmitry Fuchs on his 80th birthday with gratitude and admiration!}

\section{Introduction}

The results of this article strengthen the connection between invariants of Legendrian knots in standard contact $\R^3$ and the $2$-variable Kauffman polynomial. Relations between the $2$-variable knot polynomials (HOMFLY-PT and Kauffman) and Legendrian knot theory were first realized in the work of Fuchs and Tabachnikov \cite{FT} who observed, based on results of Bennequin \cite{Be}, Franks--Morton--Williams \cite{FW}, and Rudolph \cite{Ru}, that these polynomials provide upper bounds on the Thurston--Bennequin number of a Legendrian knot. At that time, it was still unknown whether Legendrian knots in $\R^3$ were determined up to Legendrian isotopy by their Thurston--Bennequin number, rotation number, and underlying smooth knot type (the so-called ``classical invariants'' of Legendrian knots). This question was soon resolved with the introduction of several non-classical invariants in the late 90's and early 2000's including the Legendrian contact homology algebra which is a differential graded algebra (DGA) coming from $J$-holomorphic curve theory that was constructed by Chekanov in \cite{Chekanov02} and discovered independently by Eliashberg and Hofer, and combinatorial invariants introduced by Chekanov and Pushkar \cite{ChP} defined by counting certain decompositions of front diagrams called normal rulings.\footnote{Around the same time, generating family homology invariants capable of distinguishing Legendrian links with the same classical invariants were introduced by Traynor~\cite{Tr}.} Interestingly, normal rulings were discovered independently by Fuchs in connection with augmentations of the Le\-gendrian contact homology DGA. Moreover, Fuchs again pointed toward a connection between Legendrian invariants and topological knot invariants by conjecturing in \cite{Fuchs} that a Legendrian knot should have a normal ruling if and only if the Kauffman polynomial estimate for the Thurston--Bennequin number is sharp. This conjecture was resolved affirmatively in \cite{R1} by interpreting Chekanov and Pushkar's combinatorial invariants as polynomials, and showing that the ungraded ruling polynomial, $R^1_K(z)$, of a Legendrian knot $K \subset \R^3$ arises as a specialization
\begin{gather} \label{eq:specialization}
R^1_K(z) = F_K(a,z)|_{a^{-1}=0}
\end{gather}
of the framed version of the Kauffman polynomial $F_K \in \Z\big[a^{\pm1}, z^{\pm1}\big]$; the specialization has the property that it is non-zero if and only if the Kauffman polynomial estimate for $\operatorname{tb}(K)$ is sharp. An analogous result also established in~\cite{R1} holds for the $2$-graded ruling polynomial, $R^2_K(z)$, and the HOMFLY-PT polynomial.

Initially, the Legendrian invariance of the ruling polynomials, based on establishing bijections between rulings during bifurcations of the front diagram occurring in a Legendrian isotopy, was somewhat mysterious from the point of view of symplectic topology. Building on the earlier works \cite{Fuchs, FI,Henry,NgSab, Sab}, Henry and the second author showed in~\cite{HenryRu} that the ruling polynomials are in fact determined by the Legendrian contact homology DGA, $(\mathcal{A}, \partial)$, since their specializations at $z = q^{1/2} -q^{-1/2}$ with $q$ a prime power agree with normalized counts of augmentations of $(\mathcal{A}, \partial)$ to the finite field $\mathbb{F}_q$, i.e., DGA representations from $(\mathcal{A},\partial)$ to $(\mathbb{F}_q,0)$. Thus, in the ungraded case (\ref{eq:specialization}) shows that counts of ungraded augmentations are actually topological (depending only on the underlying framed knot type of $K$), as they arise from a specialization of the Kauffman polynomial. In this article we extend this result by relating counts of higher dimensional (ungraded) representations of $(\mathcal{A}, \partial)$ with the $n$-colored Kauffman polynomials.

To give a statement of our main result, for $n \geq 1$, let $\operatorname{Rep}_1\big(K,\mathbb{F}_q^n\big)$ denote the {\it ungraded total $n$-dimensional representation number} of $K$ as defined in~\cite{LeRu}; see Definition \ref{def:total}. Let $F_{n,K}(a,q)$ denote the {\it $n$-colored Kauffman polynomial} (for framed knots); see Definition~\ref{def:Kauffman}. In Section~\ref{sec:ncolored}, we define an {\it ungraded $n$-colored ruling polynomial} denoted~$R^1_{n,K}(z)$.

\begin{Theorem} \label{thm:main} For any Legendrian knot $K$ in $\R^3$ with its standard contact structure and any $n \geq 1$, there is a well-defined specialization $F_{n,K}(a,q)|_{a^{-1}=0}$, and we have
\[
\operatorname{Rep}_1\big(K,\mathbb{F}_q^n\big)=R^1_{n,K}(z)=F_{n,K}(a,q)|_{a^{-1}=0}.
\]
\end{Theorem}

As an immediate consequence, we get:

\begin{Corollary}
The ungraded total $n$-dimensional representation number $\operatorname{Rep}_1\big(K,\mathbb{F}_q^n\big)$ depends only on the underlying framed knot type of $K$.
\end{Corollary}
The corollary is a significant strengthening of a result from~\cite{NgR2012} that the existence of an ungraded representation of~$(\mathcal{A}, \partial)$ on $\mathbb{F}_2^n$ depends only on the Thurston--Bennequin number and topological knot type of~$K$. Precisely how much of the Legendrian contact homology DGA is determined by the framed knot type of~$K$ remains an interesting question. See~\cite{NgR2012} and~\cite{Ng1} for some open conjectures along this line.

A previous article \cite{LeRu} establishes analogous results in the case of $2$-graded representations and the HOMFLY-PT polynomial, and in fact establishes the equality $\operatorname{Rep}_m\big(K,\mathbb{F}_q^n\big)=R^m_{n,K}(z)$ between the $m$-graded total representation numbers and $m$-graded colored ruling polynomials for all $m \in \Z_{\geq 0}$ {\it except for the ungraded case where $m=1$}.  The $m=1$ case is more involved for a number of reasons. In the following we briefly review the argument from~\cite{LeRu} and then outline our approach to Theorem~\ref{thm:main}.

For $m \neq 1$, the $n$-colored ruling polynomial is defined as a linear combination of satellite ruling polynomials of the form
\begin{gather} \label{eq:mgraded}
R^m_{n,K}(q) = \frac{1}{c_n}\sum_{\beta \in S_n} q^{\lambda(\beta)/2} R^m_{S(K,\beta)}(z)\big|_{z = q^{1/2}-q^{-1/2}}, \qquad m \neq 1,
\end{gather}
where \looseness=1 $S(K,\beta)$ is the Legendrian satellite of $K$ with a Legendrian positive permutation braid associated to $\beta \in S_n$. The same linear combination of HOMFLY-PT polynomials defines the $n$-colored HOMFLY-PT polynomial. In~\cite{LeRu}, the total $n$-dimensional representation number is recovered from~(\ref{eq:mgraded}) via a bijection between $m$-graded augmentations of $S(K,\beta)$ and $n$-dimensional representations of the DGA of $K$ mapping a distinguished invertible generator into $B_\beta \subset {\rm GL}(n, \mathbb{F})$ where ${\rm GL}(n,\mathbb{F}) = \sqcup_{\beta \in S_n}B_\beta$ is the Bruhat decomposition. Thus, summing over all $\beta \in S_n$ corresponds to considering all $n$-dimensional representations of the DGA of~$K$ on~$\mathbb{F}^n$.

When $m=1$, the above bijection becomes modified in an interesting way, as augmentations of $S(K,\beta)$ now correspond to (ungraded) representations of $(\mathcal{A}, \partial)$ on differential vector spaces of the form $\big(\mathbb{F}^n,d\big)$ where $d$ varies over all (ungraded) upper triangular differential on $\mathbb{F}^n$. (When $m \neq 1$, $d =0$ is automatic for grading reasons.) The total $n$-dimensional representation number, $\operatorname{Rep}_1\big(K,\mathbb{F}_q^n\big)$, only counts representations with $d=0$, so the definition of $R^1_{n,K}$ needs to be changed to only take into account normal rulings corresponding to representations with $d=0$. This is done by replacing each $R_{S(K,\beta)}^1$ in~(\ref{eq:mgraded}) with the corresponding {\it reduced ruling polynomial}~$\widetilde{R}^1_{S(K,\beta)}$ as introduced in~\cite{NgR2012} that only counts normal rulings of $S(K,\beta)$ that never pair two strands of the satellite that correspond to a single strand of $K$. Up to a technical point about the use of different diagrams in~\cite{LeRu} and~\cite{HenryRu} that the bulk of Section~\ref{sec:5} is spent addressing, this leads to the equality $\operatorname{Rep}_1\big(K,\mathbb{F}_q^n\big)=R^1_{n,K}(z)$.

Establishing that $R^1_{n,K}(z)=F_{n,K}(a,q)|_{a^{-1}=0}$ requires a much more involved argument than for the case of the colored HOMFLY-PT polynomial and $R^2_{n,K}(z)$ (where the result is immediate from~\cite{R1} and the definition). The $n$-colored Kauffman polynomial is defined by satelliting $K$ with the symmetrizer in the BMW algebra, $\mathcal{Y}_n \in \operatorname{BMW}_n$. In addition to a sum over permutation braids as in the HOMFLY-PT case, $\mathcal{Y}_n$ also has terms of a less explicit nature (though, see~\cite{D}) involving tangle diagrams in $[0,1]\times \R$ with turn-backs, i.e., components that have both endpoints on the same boundary component of $[0,1]\times \R$. To relate $R^1_{n,K}$ and $F_{n,K}$, we use the combinatorics of normal rulings to find an inductive characterization of $R^1_{n,K}$ in terms of ordinary ruling polynomials rather than reduced ruling polynomials, and then compare this with an inductive characterization of $\mathcal{Y}_n$ due to Heckenberger and Sch{\"u}ler~\cite{Heck}.

The remainder of the article is organized as follows. In Section \ref{sec:ncolored}, we define the ungraded $n$-colored ruling polynomial and establish an inductive characterization of it in Theorem \ref{thm:main1}. In Section \ref{sec:Kauffman}, we recall the definition of the colored Kauffman polynomial and prove the second equality of Theorem \ref{thm:main} (see Theorem~\ref{thm:main2}). Section \ref{sec:5} reviews definitions of representation numbers from \cite{LeRu} and then establishes the first equality of Theorem~\ref{thm:main} (see Theorem~\ref{thm:Rep}). In Section \ref{sec:multicomp}, we close the article with a brief discussion of a modification of Theorem \ref{thm:main} for the case of multi-component Legendrian links.

\section[The $n$-colored ungraded ruling polynomial]{The $\boldsymbol{n}$-colored ungraded ruling polynomial} \label{sec:ncolored}

In this section, after a brief review of ruling polynomials and Legendrian satellites, we define the ungraded $n$-colored ruling polynomial $R^1_{n,K}$ as a linear combination of {\it reduced} ruling polynomials indexed by permutations $\beta \in S_n$. Reduced rulings, considered earlier in~\cite{NgR2012}, form a restricted class of normal rulings of satellite links, so that it is not immediately clear how to describe $R^1_{n,K}$ in terms of ordinary ruling polynomials. For this purpose, we work in a Legendrian version of the $n$-stranded BMW algebra, $\operatorname{BMW}_n^\leg$, and inductively construct elements $L_n \in \operatorname{BMW}_n^\leg$ that can be used to produce $R^1_{n,K}$ via (non-reduced) satellite ruling polynomials.

\subsection{Legendrian fronts and ruling polynomials}
In this article we consider Legendrian links and tangles in a $1$-jet space $J^1M$ where $M$ is one of~$\R$,~$S^1$, or~$[0,1]$. In all cases, we can view $J^1M = T^*M\times \R$ as $M \times \R^2$, and using coordinates $(x,y,z)$ with $x \in M$ and $y,z \in \R$ the contact form is ${\rm d}z - y\,{\rm d}x$. Legendrian curves can be viewed via their {\it front projection} $\pi_{xz}\colon J^1M \rightarrow M \times \R$, $(x,y,z) \mapsto (x,z)$ which is a collection of curves having cusp singularities and transverse double points but no vertical tangencies. The original Legendrian is recovered via $y= \frac{{\rm d}z}{{\rm d}x}$, so in front diagrams (implicitly) the over-strand at a crossing is the strand with lesser slope (as the $y$-axis is oriented away from the viewer). Legendrian links have a {\it contact framing} which is the framing given by the upward unit normal vector to the contact planes.

\begin{figure}[t] \centering
 \includegraphics{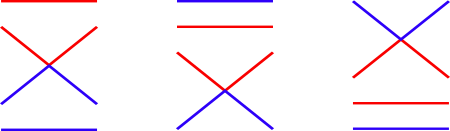}
 \caption{Each closed curve of a normal ruling consists of a pair of companion paths with monotonically increasing $x$-coordinate beginning and ending at a common left and right cusp of~$\pi_{xz}(L)$. At switches, paths from two different closed curves of~$\rho$ meet and both turn a corner at a crossing. The normality condition requires that near switches the switching paths and their companion paths match one of the pictured configurations. At crossings that are not switches paths from two different closed curves cross transversally.} \label{fig:normality}
\end{figure}

\begin{figure}[h!] \centering
 \begin{align*}
 \raisebox{-.5cm}{\includegraphics[scale=.7]{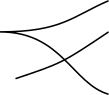}}-\raisebox{-.5cm}{\includegraphics[scale=.7]{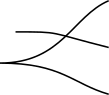}} &=z\left( \;\raisebox{-.5cm}{\includegraphics[scale=.7]{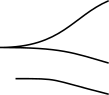}}-\; \raisebox{-.5cm}{\includegraphics[scale=.7]{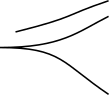}}\right), \tag{R1} \\
 \raisebox{-.5cm}{\includegraphics[scale=.7]{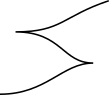}}& =\raisebox{-.5cm}{\includegraphics[scale=.7]{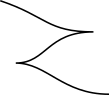}}=0, \tag{R2} \\
 \raisebox{-.5cm}{\includegraphics[scale=.7]{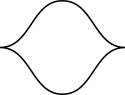}} \;\sqcup K &=z^{-1}K. \tag{R3}
 \end{align*}\vspace{-0.7cm}
 \caption{The ungraded ruling polynomial skein relations.}\label{fig:psuedo-gen}
\end{figure}

Recall that for a Legendrian link $K \subset J^1\R$ a {\it normal ruling} $\rho$ of $K$ is a decomposition of the front diagram of $K$ into a collection of simple closed curves with corners at a left and right cusp and at {\it switches} (adhering to the normality condition, see Fig.~\ref{fig:normality}). For each $x=x_0$ where the front projection of $K$ does not have crossings or cusps, a normal ruling $\rho$ divides the strands of~$K$ at $x=x_0$ into pairs, so that $\rho$ can be viewed as a sequence of pairings of strands of~$K$. For a~more detailed discussion of normal rulings see for instance \cite{Fuchs,NgR2012, R1, Sab}. The {\it ungraded ruling polynomial} of $L$ (also called the $1$-graded ruling polynomial) is defined as
\[
R_K^1(z):=\sum_{\rho\in\Gamma(K)}z^{j(\rho)} \in \Z\big[z^{\pm1}\big],
\] where the sum is over all normal rulings of $K$ and $j(\rho)=\#\text{switches}-\#\text{right cusps}$. For Legendrian links in $J^1\R$, the ungraded ruling polynomial of $K$ satisfies and is uniquely characterized by the skein relations in Fig.~\ref{fig:psuedo-gen} and the normalization $R^1_{\includegraphics[scale=.2]{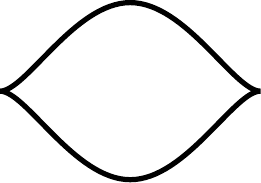}} = z^{-1}$. (See~\cite{R1}.) The relations in Fig.~\ref{fig:psuedo-gen} imply two additional relations that we will make use of, cf.\ \cite[Section~6]{R2}.
\begin{gather} \label{eq:fish}
 \text{fishtail relation:} \ \raisebox{-.35cm}{\includegraphics[scale=.5]{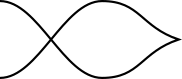}} =\raisebox{-.35cm}{\includegraphics[scale=.5]{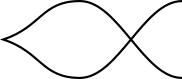}}=0,
 \vspace{.5cm}\\
 \text{double-crossing relation:} \ \raisebox{-.35cm}{\includegraphics[scale=.5]{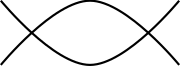}}= \raisebox{-.35cm}{\includegraphics[scale=.5]{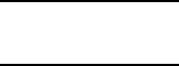}}+ z\cdot\raisebox{-.35cm}{\includegraphics[scale=.5]{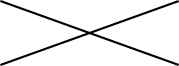}}- z\cdot\raisebox{-.35cm}{\includegraphics[scale=.5]{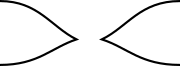}}.\nonumber
\end{gather}
\begin{Remark} The double crossing relation can be used to show that the first relation of Fig.~\ref{fig:psuedo-gen} also hold with right cusps. Moreover, the third relation from Fig.~\ref{fig:psuedo-gen} is implied by the first two as long as $K \neq \varnothing$.
\end{Remark}

\subsection{A Legendrian BMW algebra}

A {\it Legendrian $n$-tangle} is a properly embedded Legendrian $\alpha \subset J^1[0,1]$ (i.e., compact with $\partial \alpha \subset \partial J^1[0,1]$) whose front projection agrees with the collection of horizontal lines $z=i$, $1\leq i \leq n$, near $\partial J^1[0,1]$.  Legendrian isotopies of $n$-tangles are required to remain fixed in a~neighborhood of $\partial J^1[0,1]$. At $x=0$ and $x=1$ we enumerate the endpoints and strands of a Legendrian tangle from $1$ to $n$ with {\it descending} $z$-coordinate. For any permutation $\beta \in S_n$, there is a corresponding {\it positive permutation braid} that is a Legendrian $n$-tangle, also denoted $\beta \subset J^1[0,1]$, that connects endpoint $i$ at $x=0$ to endpoint $\beta(i)$ at $x=1$ for $1 \leq i \leq n$. Up to Legendrian isotopy, $\beta$ is uniquely characterized by requiring that
\begin{itemize}\itemsep=0pt
\item[(i)] the front projection does not have cusps and,
\item[(ii)] for $i<j$, the front projection has no crossings (resp.\ exactly one crossing) between the strands with endpoints $i$ and $j$ at $x=0$ if $\beta(i) < \beta(j)$ (resp.\ $\beta(i)> \beta(j)$).
\end{itemize}
The number of crossings in such a front diagram for $\beta$ is called the {\it length} of $\beta$ and will be deno\-ted~$\lambda(\beta)$. For Legendrian $n$-tangles $\alpha,\beta \subset J^1[0,1]$, we define their multiplication $\alpha\cdot\beta \subset J^1[0,1]$ by stacking $\beta$ to the {\it left} of $\alpha$ (as in composition of permutations). Diagrammatically:
\[
\labellist
\small
\pinlabel $\alpha$ at 45 38
\endlabellist
\includegraphics[scale=.5]{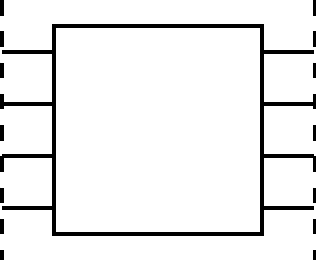} \,\, \raisebox{.6cm}{$\cdot$} \,\,
\labellist
\small
\pinlabel $\beta$ at 45 38
\endlabellist
\includegraphics[scale=.5]{images/Tangle1}
\quad\raisebox{.6cm}{$=$} \quad
\labellist
\small
\pinlabel $\beta$ at 45 38
\pinlabel $\alpha$ at 120 38
\endlabellist
\includegraphics[scale=.5]{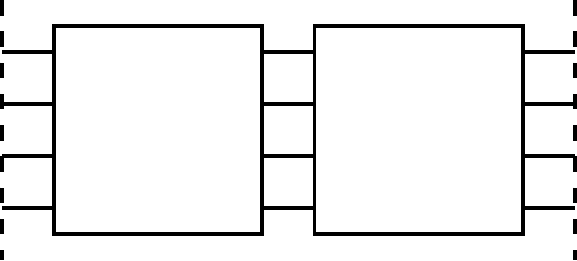}
\]

\begin{Definition} Let $\mathcal{R}$ be a coefficient ring containing $\Z[z^{\pm1}]$ as a subring. Define the {\it Legendrian BMW algebra}, $\operatorname{BMW}_n^\leg$, as an $\mathcal{R}$-module to be the quotient $\mathcal{R} \leg_n/\mathcal{S}$ where $\mathcal{R} \leg_n$ is the free $\mathcal{R}$-module generated by Legendrian isotopy classes of Legendrian $n$-tangles in $J^1[0,1]$, and $\mathcal{S}$ is the $\mathcal{R}$-submodule generated by the ruling polynomial skein relations from Fig.~\ref{fig:psuedo-gen}. Multiplication of $n$-tangles induces an $\mathcal{R}$-bilinear product on $\operatorname{BMW}_n^\leg$.
\end{Definition}

In the remainder of the article, we fix the coefficient ring $\mathcal{R}$ to be $\Z\big[s^{\pm1}\big]$ localized to include denominators of the form $s^n-s^{-n}$ for $n\geq 1$ where $z = s- s^{-1}$. In Section~\ref{sec:5}, we will work with the alternate variable $s= q^{1/2}$.

\begin{figure} \centering
 $\sigma_i:=\raisebox{-.75cm}{\includegraphics[scale=.5]{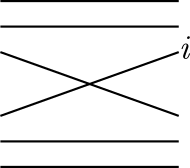}}\hspace{.75cm} \hspace{.75cm} e_i:=\raisebox{-.75cm}{\includegraphics[scale=.5]{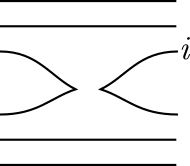}}$
 \caption{Crossing and hook elements in $\operatorname{BMW}_n^\leg$.} \label{fig:sigma}
\end{figure}

Fig.~\ref{fig:sigma} indicates crossing and hook elements, $\sigma_i, e_i \in \operatorname{BMW}_n^\leg$, for $1 \leq i < n$. Note that the fishtail and double-crossing relations imply
\begin{gather}
 \sigma_i e_i = e_i \sigma_i =0, \qquad \mbox{and} \label{eq:fish2} \\
 \sigma_i^2 = 1 + z \sigma_i - z e_i, \qquad \mbox{for $1 \leq i <n$.} \label{eq:double}
\end{gather}

\begin{Remark} \label{rem:BMW}\samepage\quad
\begin{enumerate}\itemsep=0pt
\item Occasionally, we will multiply an element $\alpha \in \operatorname{BMW}_i^\leg$ by an element of $\operatorname{BMW}_j^\leg$ with $i<j$. Unless indicated otherwise, we do so by extending $i$-tangles to $j$-tangles by placing $j-i$ horizontal strands below $\alpha$.

\item We do not give a complete set of generators and relations for $\operatorname{BMW}_n^\leg$, as the main role of $\operatorname{BMW}_n^\leg$ in this article is to provide a convenient setting for ruling polynomial calculations. We leave finding a presentation for $\operatorname{BMW}_n^\leg$ as an open problem.
\end{enumerate}
\end{Remark}

\subsection{Legendrian satellites and reduced ruling polynomials} \label{sec:sat} We will be considering rulings of satellites. Given a connected Legendrian knot $K \subset J^1\R$ and a Legendrian $n$-tangle $L \subset J^1[0,1]$, a~front diagrammatic description of the {\it Legendrian satellite} $S(K,L) \subset J^1\R$ is the following: Form the $n$-copy of~$K$, i.e., take $n$ copies of $K$ each shifted up a~small amount in the $z$-direction, then insert a rescaled version of the $n$-tangle $L$ into the $n$-copy at a small rectangular neighborhood $J \cong [0,1] \times [-\epsilon, \epsilon]$ of part of a strand of $K$ that is oriented from left-to-right. We refer to $J$ as the {\it $J^1[0,1]$-part} of the satellite. See Fig.~\ref{fig:Sat}. It can be shown that the Legendrian isotopy type of~$S(K,L)$ depends only on~$K$ and the closure of $L$ in $J^1S^1$, cf.~\cite{NgTraynor}. (See also Section~\ref{sec:satellite} for additional discussion.)

\begin{figure} \centering
		\labellist
\small
\pinlabel $K$ [t] at 36 30
\pinlabel $\beta$ [t] at 170 44
\pinlabel $S(K,\beta)$ [t] at 352 0
\endlabellist

 \includegraphics{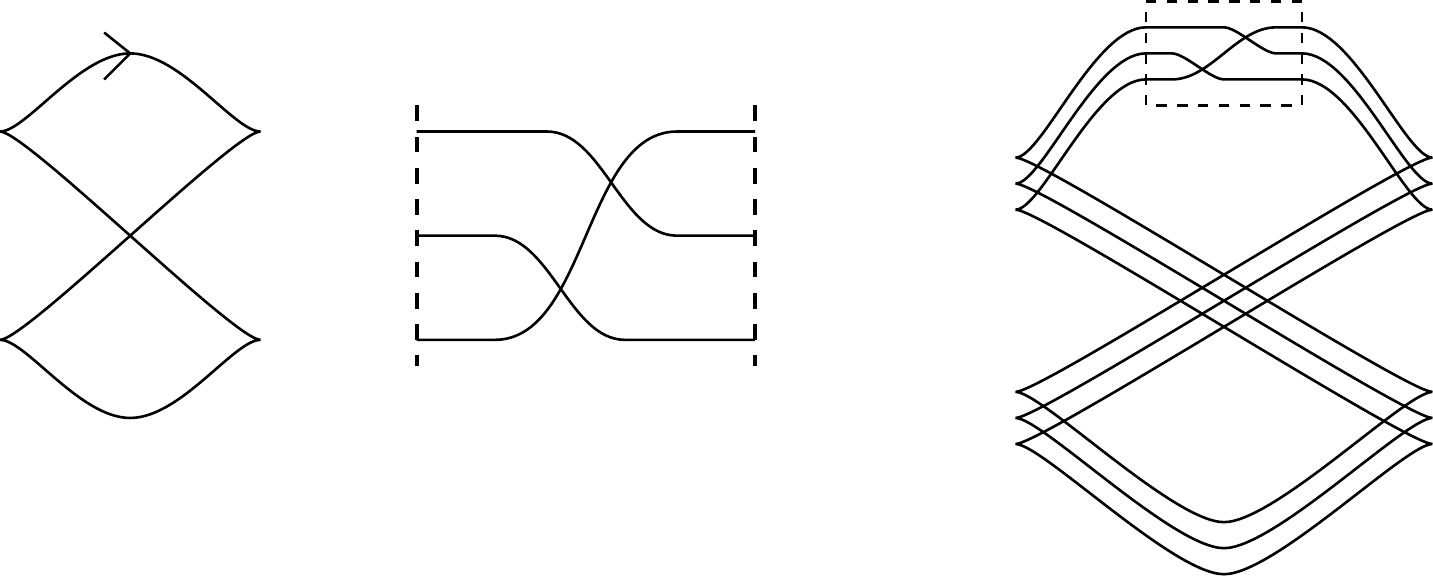}
		
\caption{The Legendrian satellite $S(K,\beta)$ where $\beta = \sigma_1 \sigma_2 \subset J^1[0,1]$ is the positive permutation braid associated to the $3$-cycle $(1 \, 2 \,3)$. The $J^1[0,1]$-part of the satellite is indicated by the dotted rectangular box.} \label{fig:Sat}
\end{figure}

In \cite{NgR2012} a variant of the ruling polynomial was introduced for satellites using reduced normal rulings. Let $L \subset J^1[0,1]$ be a Legendrian $n$-tangle. A normal ruling $\rho$ of $S(K,L)$ is said to be {\it reduced} if, outside of the $J^1[0,1]$-part of $S(K,L)$, parallel strands of $S(K,L)$ corresponding to a~single strand of~$K$ are not paired by $\rho$. See Fig.~\ref{fig:reduced}.
\begin{figure}[h] \centering
 \includegraphics{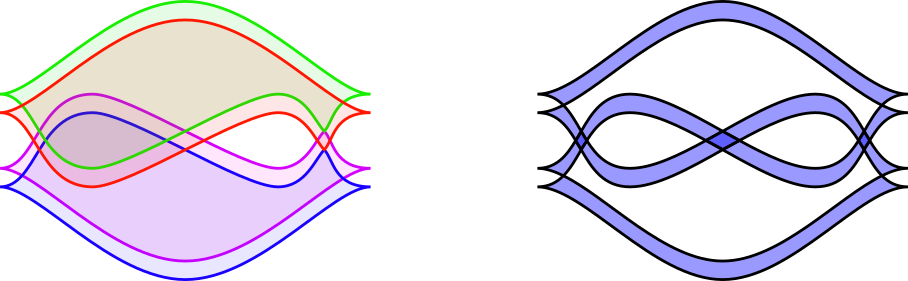}
 \caption{A reduced normal ruling (left) and a non-reduced normal ruling (right) of $S(K,\beta)$ where $K$ is a right-handed trefoil and $\beta=1$.} \label{fig:reduced}
\end{figure}
 We denote the set of reduced normal rulings of $S(K,L)$ by $\wt{\Gamma}(K,L)$ and define the {\it reduced ruling polynomial} of $S(K,L)$ as $\wt{R}_{S(K,L)}(z):=\sum\limits_{\rho\in\wt{\Gamma}(K,L)}z^{j(\rho)}$.

\begin{Remark} \label{rem:thin}If the reduced condition holds for a normal ruling $\rho$ of $S(K,L)$ for the parallel strands of $S(K,L)$ corresponding to a single point $k_0 \in K$ outside of the $J^1[0,1]$-part of $S(K,L)$, then $\rho$ will be reduced. Indeed, the involution of parallel strands of $S(K,L)$ at $k_0 \in K$ that arises from restriction the pairing of $\rho$ (strands of $S(K,L)$ at $k_0$ that are paired with strands in a part of the satellite away from~$k_0$ are fixed points of the involution) is called the ``thin part'' of~$\rho$ at~$k_0$ in~\cite{NgR2012}, and Lemmas~3.3 and~3.4 from~\cite{NgR2012} show that the thin part is independent of~$k_0$.
\end{Remark}

We are now prepared to make the following key definition.

\begin{Definition} \label{def:nruling}For any (connected) Legendrian knot $K\subset J^1\R$ we define the {\it ungraded $n$-colored ruling polynomial} as
\begin{gather*}
 R_{n,K}^1(s):=\frac{1}{c_n}\sum_{\beta\in S_n}s^{\lambda(\beta)}\wt{R}_{S(K,\beta)}(z),
\end{gather*}
where the sum is over positive permutation braids, $z=s-s^{-1}$, $\lambda(\beta)$ is the length of $\beta$, and $c_n=s^{n(n-1)/2}\prod\limits_{i=1}^n\frac{s^i-s^{-i}}{s-s^{-1}}$.
 \end{Definition}
Note that $R_{n,K}^1$ belongs to coefficient ring $\mathcal{R}$ defined above.

\begin{Remark}It is proved in \cite{NgR2012} that with $L$ fixed the reduced ruling polynomial $\widetilde{R}_{S(K,L)}(z)$ is a Legendrian isotopy invariant of~$K$. Alternatively, the Legendrian isotopy invariance of $R_{n,K}^1$ also follows from Theorem~\ref{thm:main1} below and invariance of the ordinary ruling polynomials.
\end{Remark}

\subsection[Inductive characterization of $R_{n,K}^1$]{Inductive characterization of $\boldsymbol{R_{n,K}^1}$} \label{sec:Ln}

It is convenient to extend the concept of satellite ruling polynomials slightly by defining $R^1_{S(K,\eta)}$ for $\eta \in \operatorname{BMW}_n^{\leg}$ represented as an $\mathcal{R}$-linear combination of $n$-tangles $\sum\limits_{i=1}^k r_i \alpha_i$ to be
\begin{gather} \label{eq:BMWsatellite}
R^1_{S(K,\eta)} = \sum_{i=1}^kr_i R^1_{S(K,\alpha_i)}.
\end{gather}
(This is well-defined since the front diagram of each $\alpha_i$ appears as a~subset of $\pi_{xz}(S(K,\alpha_i))$ and~$R^1$ satisfies the skein relations that define $\operatorname{BMW}_n^\leg$.) The goal in the remainder of this section is to give such a characterization of the ungraded $n$-colored ruling polynomials via appropriate elements $L_n \in \operatorname{BMW}_n^\leg$.

Let
\[
\gamma_n=1+\sum_{j=1}^{n-1}s^j\sigma_{n-1}\sigma_{n-2}\cdots \sigma_{n-j} \in \operatorname{BMW}^\leg_n,
\]
 where $\sigma_i$ is as in Fig.~\ref{fig:sigma} and consider the elements of $\operatorname{BMW}^\leg_n$
 defined inductively for $n\geq1$ by $L_1 =1$ and
 \[L_n=L_{n-1}\beta_n, \qquad n \geq 2,\]
where
\[ \beta_n:=\left(1-z\sum_{k=2}^n\alpha_{k,n}\right)\gamma_n.\]
Here, $\alpha_{k,n}\in \operatorname{BMW}^\leg_n$ is defined inductively on the {\it bottom} $k$ strands by $\alpha_{2,n} = e_{n-1}$ and
 \begin{gather} \alpha_{k,n}=\raisebox{-.25cm}{\includegraphics{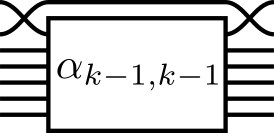}}-z\cdot\raisebox{-.25cm}{\includegraphics{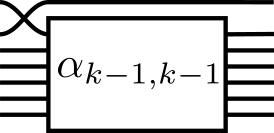}}-z\cdot\raisebox{-.25cm}{\includegraphics{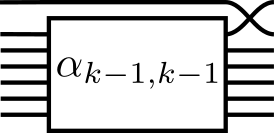}}+z^2\cdot\raisebox{-.25cm}{\includegraphics{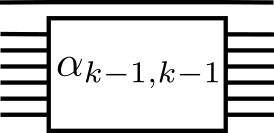}} \label{alpha}\end{gather}
		with the identity tangle appearing as the top $n-k$ strands of $\alpha_{k,n}$. For example, using (\ref{eq:fish2}) we have $L_2=\mathbf{1}+s\sigma_1-ze_1$.
		The double subscript notation on $\alpha_{k,n}$ is to emphasize that $\alpha_{k,n}$
		involves the bottom $k$ strands, $n-k+1, n-k+2, \ldots, n$, out of $n$ total strands.
		(By contrast, when multiplying $L_{n-1} \beta_n$ an extra strand is placed at the bottom of $L_{n-1}$; see Remark~\ref{rem:BMW}.)

In Section~\ref{sec:Kauffman}, we will also make use of a non-inductive characterization of the~$\alpha_{k,n}$. Given a~front diagram~$D$, let $\operatorname{cr}(D)$ denote the set of crossings appearing in $D$. We define the {\it resolution of $D$ with respect to $X\subset \operatorname{cr}(D)$} to be the tangle $r_X(D)$ obtained from resolving all crossings of~$X$ as depicted in Fig.~\ref{fig:res}.

 \begin{figure}[h] \centering
 \includegraphics[scale=.4]{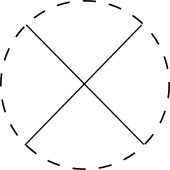}\raisebox{.9cm}{$\;\longrightarrow$} \includegraphics[scale=.4]{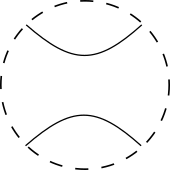}
 \caption{The resolution of a crossing in $X\subset \operatorname{cr}(D)$.} \label{fig:res}
 \end{figure}
 \begin{Lemma} For $1<k\leq n$, \[\raisebox{.6cm}{$\alpha_{k,n}=\sum\limits_{X\subset \operatorname{cr}(C_{k,n})} (-z)^{\vert X \vert} r_X(C_{k,n})$, \quad where $C_{k,n}=$ }\includegraphics[scale=.65]{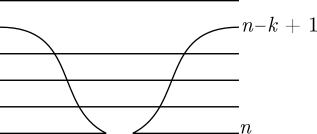}\]\label{lemma:alpha}
 \end{Lemma}
		
\begin{proof}This is a straightforward induction on~$k$.
\end{proof}

\begin{Theorem}\label{thm:main1}For any Legendrian $K\subset J^1\R$, $R_{n,K}^1=\frac{1}{c_n}R^1_{S(K,L_n)}$.
\end{Theorem}

The proof will be given below after Lemma~\ref{lemma:THMB}.

Given $K\subset J^1\R$, and $\eta$ a Legendrian $n$-tangle, recall that the $J^1[0,1]$-part of $S(K,\eta)$ refers to a rectangular region, $J\cong [0,1] \times [-\epsilon,\epsilon]$, where the front diagram of $\eta$ appears. We will call a normal ruling~$\rho$ of~$S(K,\eta)$ {\it $k$-reduced} if {\it at the right boundary of~$J$} none of the top~$k$ parallel strands of $\eta$ within $S(K,\eta)$ are paired with one another. We denote by $\wt{\Gamma^k}(K,\eta)$, the set of $k$-reduced normal rulings of~$S(K,\eta)$. Similarly, given a~location $*$ within $J$ corresponding to an $x$-value in $J^1[0,1]$ without double points of $\eta$ a~normal ruling $\rho$ of $S(K,\eta)$ is said to be {\it $(k,m)$-paired at $*$} if $\rho$ is $(m-1)$-reduced (at the right boundary of~$J$) and $\rho$ pairs strand $k$ of~$\eta$ with strand~$m$ of $\eta$ at $*$ (where strands are numbered from top to bottom as they appear in~$\eta$ at~$*$). The set of $(k,m)$-paired rulings of $S(K,\eta)$ is denoted by $\Gamma_{*}^{\tau(k,m)}(K,\eta)$. The {\it $k$-reduced ruling polynomial} and {\it $(k,m)$-paired ruling polynomial} (at $*$) are defined respectively by \[ \wt{R}^{(k)}_{(K,\eta)}(z)=\sum_{\rho\in\wt{\Gamma^k}(K,\eta)}z^{j(\rho)} \qquad \text{and} \qquad R^{\tau(k,m)}_{(*;K,\eta)}(z)=\sum_{\rho\in\Gamma_{*}^{\tau(k,m)}(K,\eta)}z^{j(\rho)}.\]
Either polynomial can be extended by linearity to allow $\eta$ to be an $\mathcal{R}$-linear combination of Legendrian $n$-tangles.

\begin{Remark}\label{rem:reduced}If $\eta$ is a $n$-stranded positive braid, then any $n$-reduced normal ruling is reduced. This is because if parallel strands of the satellite $S(K,\eta)$ corresponding to a single strand of $K$ are not paired at the right side of the $J^1[0,1]$-part of the satellite, then such parallel strands cannot be paired anywhere outside of the $J^1[0,1]$-part. See Remark~\ref{rem:thin}.
\end{Remark}

\begin{Lemma}\label{lemma:THMB} For any $1<k \leq m \leq n$, let $\beta \in S_{m-1}$ be a positive permutation braid extended to an Legendrian $n$-tangle by placing the $n-m+1$ stranded identity tangle below $\beta$, and let $\nu$ be a linear combination of Legendrian $n$-tangles. Then, we have
\[z\wt{R}^{(m-1)}_{(K,\beta\alpha_{k,m}\nu)}(z)=R^{\tau(m-k+1,m)}_{(*;K,\beta\nu)}(z),\]
where $*$ is the location between $\beta$ and $\nu$ $($immediately to the left diagrammatically of~$\beta)$.
\end{Lemma}

\begin{proof}The proof is by induction on $k$ with $m$ and $n$ fixed. The base case is clear since $\alpha_{2,m}=e_{m-1}$.

Suppose the statement is true for~$k$. For clarity of this proof, we emphasize the location of our paired rulings within our notation, and abbreviate $\alpha_{k,m}$ as~$\alpha_k$. Using~(\ref{alpha}) and the inductive hypothesis, we compute
\begin{gather*}
 z\wt{R}^{(m-1)}_{(K,\beta\alpha_{k+1}\nu)}(z)\\
 =z\wt{R}^{(m-1)}_{(K,\beta\sigma_{m-k}\alpha_{k}\sigma_{m-k}\nu)}(z) -z^2\wt{R}^{(m-1)}_{(K,\beta\sigma_{m-k}\alpha_{k}\nu)}(z) -z^2\wt{R}^{(m-1)}_{(K,\beta\alpha_{k}\sigma_{m-k}\nu)}(z)+z^3\wt{R}^{(m-1)}_{(K,\beta\alpha_{k}\nu)}(z)\\
 =R^{\tau(m-k+1,m)}_{(*;K,\beta\sigma_{m-k}*\sigma_{m-k}\nu)}(z) -zR^{\tau(m-k+1,m)}_{(*;K,\beta\sigma_{m-k}*\nu)}(z) -zR^{\tau(m-k+1,m)}_{(*;K,\beta*\sigma_{m-k}\nu)}(z)+z^2R^{\tau(m-k+1,m)}_{(*;K,\beta*\nu)}(z) \hypertarget{star}{}\\
 \overset{\star}{=}R^{\tau(m-k,m)}_{(*;K,\beta*\nu)}(z).
\end{gather*}
We now establish equality $\star$. For the diagrams $D$ that follow, when considering $(k,m)$-paired rulings, we indicate the location $*$ with a green vertical segment having endpoints on the paired strands~$k$ and~$m$. When we encode such information in $D$ we suppress the paired notation $R^{\tau(k,m)}_{(*;K,D)}$ as $R^\tau_{K,D}$ (and similarly for sets of rulings: $\Gamma_{*}^{\tau(k,m)}(K,D)$ becomes $\Gamma^{\tau}(K,D)$). For instance,
\[
R^{\tau(m-k,m)}_{(*;K,\beta*\nu)}(z) =
 R^\tau_{\left(K,\; \raisebox{-.25cm}{\includegraphics[scale=.6]{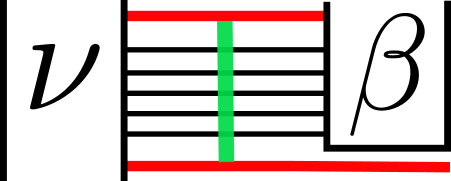}}\right)}(z),
\]
where the diagram only pictures strands with numberings in the range $m-k$ to $m$ at $*$. Observe the bijection
\begin{gather}\label{eq:color}
 \Gamma^\tau\left(K, \raisebox{-.25cm}{\includegraphics[scale=.6]{images/paired1.png}}\right)=\left\{\rho \in \Gamma^\tau\left(K, \raisebox{-.25cm}{\includegraphics[scale=.6]{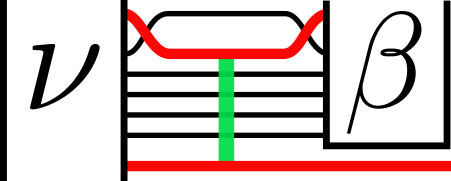}}\right) \,\bigg|\, \mbox{$r_1$ and $r_2$ are not switches of $\rho$}\right\},
\end{gather}
 where $r_1$ and $r_2$ denote the left and right crossings of the diagram on the right. In addition, we have the following (separate) bijections:
\begin{gather}\label{eq:bijectionsec2}
\begin{array}{@{}lrcl}
 &\rho&\mapsto &\Phi(\rho),\\
 \Big\{ \rho\in\Gamma^\tau\left(K,\raisebox{-.25cm}{\includegraphics[scale=.65]{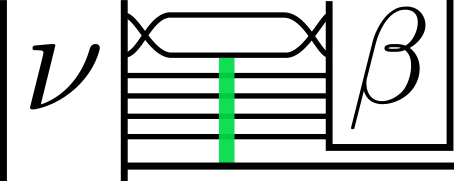}}\right) \,\big|\,& \mbox{ $r_1$ is a switch} \Big\} &{\longleftrightarrow}
 & \Gamma^\tau\left(K,\raisebox{-.25cm}{\includegraphics[scale=.65]{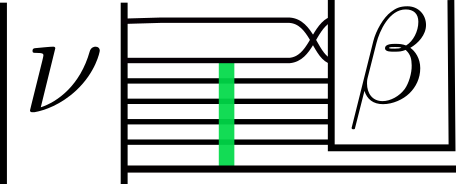}}\right), \vspace{1mm}\\
 \Big\{\rho\in\Gamma^\tau\left(K,\raisebox{-.25cm}{\includegraphics[scale=.65]{images/paired3.png}}\right) \, \big|\, & \mbox{$r_2$ is a switch}\Big\} &
 \longleftrightarrow &
 \Gamma^\tau\left(K,\raisebox{-.25cm}{\includegraphics[scale=.65]{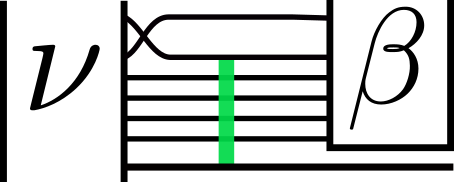}}\right), \vspace{1mm}\\
 \Big\{\rho\in\Gamma^\tau\left(K,\raisebox{-.25cm}{\includegraphics[scale=.65]{images/paired3.png}}\right) \,\big|\, & \mbox{$r_1$ and $r_2$ are switches }\Big\} & \longleftrightarrow & \Gamma^\tau\left(K,\raisebox{-.25cm}{\includegraphics[scale=.65]{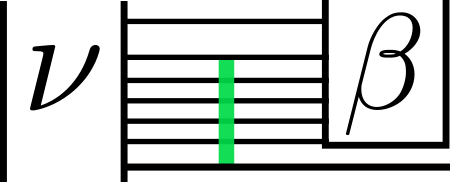}}\right) ,
 \end{array}
\end{gather}
where $\Phi(\rho)$ is the unique ruling that agrees with $\rho$ outside of a neighborhood of the resolved crossing(s). (Any ruling of a diagram $D$ with switches at a set of crossings~$X$ gives rise to a~ruling of $r_{X}(D)$.) Each $\Phi$ is clearly injective. To verify surjectivity, observe that any ruling $\rho'$ in one of the sets on the right side of~(\ref{eq:bijectionsec2}) indeed comes from a ruling on the left side by replacing the resolved crossing(s) with switches. Here, it is crucial to note that the normality condition is automatically satisfied. If the normality condition was not satisfied, then $\rho'$ must pair strand $m-k$ with a strand numbered in the range $m-k+2$ to $m-1$ at the location~$*$ indicated by the vertical segment. Since $\beta \in S_{m-1}$, it follows that at least~2 of the top $m-1$ strands are paired at the right side of the $J^1[0,1]$-part of the satellite, contradicting that $\rho'$ is $(m-1)$-reduced. Finally, the inclusion-exclusion principle and the above bijections~\eqref{eq:color} and~\eqref{eq:bijectionsec2} imply
\begin{gather*}
 R^\tau_{\left(K, \;
 \raisebox{-.25cm}{\includegraphics[scale=.54]{images/paired1.png}}\right)}(z)\\
 =R^\tau_{\left(K,\; \raisebox{-.25cm}{\includegraphics[scale=.54]{images/paired3.png}}\right)}(z)
 -zR^\tau_{\left(K, \; \raisebox{-.25cm}{\includegraphics[scale=.54]{images/paired4.png}}\right)}(z)
 -zR^\tau_{\left(K,\; \raisebox{-.25cm}{\includegraphics[scale=.54]{images/paired5.png}}\right)}(z)
 +z^2R^\tau_{\left(K,\; \raisebox{-.25cm}{\includegraphics[scale=.54]{images/paired6.png}}\right)}(z).
\end{gather*}
(The factors of $z$ and $z^2$ appear since the first two bijections in (\ref{eq:bijectionsec2}) decrease the number of switches by $1$, while the last bijection decreases the number by 2.) This establishes \hyperlink{star}{$\star$}, so we are done.
\end{proof}

\begin{proof}[Proof of Theorem \ref{thm:main1}] With $n\geq 1$ fixed we prove the following statement by induction on~$m$. The theorem follows from the special case where $m=n$ and $\mu = 1$ (in view of Remark~\ref{rem:reduced}).

{\bf Inductive statement:} For any (linear combination of) Legendrian $n$-tangle(s) $\mu \subset J^1[0,1]$
 and any $1 \leq m \leq n$,
\[
R_{S(K,L_{m}\mu)}(z)=\sum_{\beta\in S_{m}}s^{\lambda(\beta)}\wt{R}^{(m)}_{(K,\beta\mu)}(z)
\]
(with the products $L_m\mu$ and $\beta\mu$ formed as in Remark~\ref{rem:BMW}).

The base case of $m=1$ is immediate since any normal ruling is $1$-reduced. For the inductive step, with $m\geq 2$ and the statement assumed for $m-1$ we compute
\begin{align*}
 \notag R_{S(K,L_m\mu)}(z)&=R_{S(K,L_{m-1}(1-z\sum\limits_{k=2}^m\alpha_{k,m})\gamma_m\mu)}(z)\\
 \notag &=R_{S(K,L_{m-1}\gamma_m\mu)}-z\sum_{k=2}^m R_{S(K,L_{m-1}\alpha_{k,m}\gamma_m\mu)}(z)\\
 \notag &=\sum_{\beta\in S_{m-1}}s^{\lambda(\beta)}\left(\wt{R}^{(m-1)}_{(K,\beta\gamma_m\mu)}(z) -z\sum_{k=2}^m\wt{R}^{(m-1)}_{(K,\beta\alpha_{k,m}\gamma_m\mu)}(z)\right)&&(\text{inductive hyp})\\
 &=\sum_{\beta\in S_{m-1}}s^{\lambda(\beta)}\left(\wt{R}^{(m-1)}_{(K,\beta\gamma_m\mu)}(z) -\sum_{k=2}^mR^{\tau(m-k+1,m)}_{(*;K,\beta*\gamma_m\mu)}(z)\right). &&(\text{Lemma~\ref{lemma:THMB}})
	\end{align*}
Note that since there are no crossings in $\beta$ involving the $m$-th strand, an $(m-1)$-reduced ruling is $m$-reduced unless at $*$ strand $m$ is paired
 with one of the strands numbered $1, \ldots, m-1$ in the $J^1[0,1]$-part of the satellite, i.e., a strand numbered $m-k+1$ with $2 \leq k \leq m$. Thus, continuing the above computation, we have
\begin{equation*}
 \sum_{\beta\in S_{m-1}}s^{\lambda(\beta)}\left(\wt{R}^{(m-1)}_{(K,\beta\gamma_m\mu)}(z)-\sum_{k=2}^mR^{\tau(m-k+1,m)}_{(*;K,\beta*\gamma_m\mu)}(z)\right)=\sum_{\beta\in S_{m-1}}s^{\lambda(\beta)}\wt{R}^{(m)}_{(K,\beta\gamma_m\mu)}.
\end{equation*}
Next note that any $\beta\in S_m$ can be written uniquely as $\beta=\beta'\sigma_{m-1}\cdots \sigma_{m-j}$ with $\beta'\in S_{m-1}$ and $0 \leq j \leq m-1$. (Here, $m-j = \beta^{-1}(m)$.) Moreover, such a factorization realizes $\beta$ as a~positive permutation braid, so that $\lambda(\beta)=\lambda(\beta')+j$. Using the definition of~$\gamma_m$, we see that
\[\sum_{\beta\in S_{m-1}}s^{\lambda(\beta)}\wt{R}^{(m)}_{(K,\beta\gamma_m\mu)}=\sum_{\beta\in S_m}s^{\lambda(\beta)}\wt{R}^{(m)}_{(K,\beta\mu)}.\] This completes the inductive step.
\end{proof}

\section{Relation to the Kauffman polynomial} \label{sec:Kauffman}
 We begin this section with a review of the definition of the $n$-colored Kauffman polynomial~$F_{n,K}$ via satelliting with the symmetrizer $\mathcal{Y}_n \in \operatorname{BMW}_n$ in the BMW algebra. Then, using an inductive characterization of $\mathcal{Y}_n$ from~\cite{Heck}, combined with the earlier characterization of $R^1_{n,K}$ we show how to recover $R^1_{n,K}$ as a specialization $F_{n,K}\big\vert_{a^{-1}=0}$. This is accomplished by relating the Legendrian~BMW element $L_n \in \operatorname{BMW}^\leg_{n}$ from Theorem~\ref{thm:main1} to~$\mathcal{Y}_n$ in a certain sub-quotient of~$\operatorname{BMW}_n$ defined in terms of Legendrian tangles. Along the way we pause to make a conjecture relating ungraded ruling polynomial skein modules with Kauffman skein modules in general.

\subsection{The Kauffman polynomial and the BMW algebra} Recall that the (framed) Kauffman polynomial (Dubrovnik version) is a regular isotopy invariant that assigns a Laurent polynomial to framed links $L\subset\R^3$ characterized by the skein relations (shown with blackboard framing)
\begin{gather*}
 \raisebox{-.5cm}{\includegraphics[scale=.6]{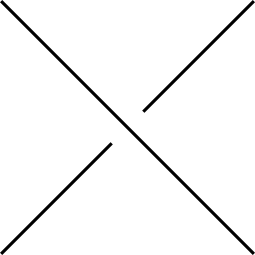}} -\raisebox{-.5cm}{\includegraphics[scale=.6]{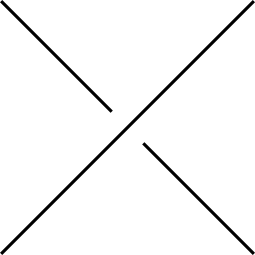}}=z\left(\; \raisebox{-.5cm}{\includegraphics[scale=.6]{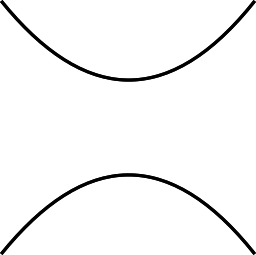}}- \raisebox{-.5cm}{\includegraphics[scale=.6]{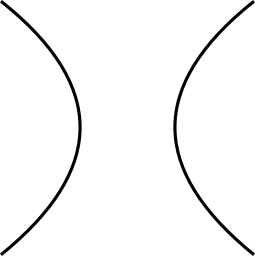}}\;\right), \tag{F1} \\
 \raisebox{-.5cm}{\includegraphics[scale=.6]{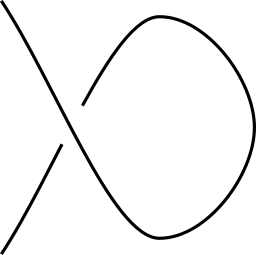}}=a^{-1}\raisebox{-.5cm}{\includegraphics[scale=.6]{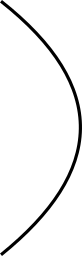}} \hspace{1cm} \raisebox{-.5cm}{\includegraphics[scale=.6]{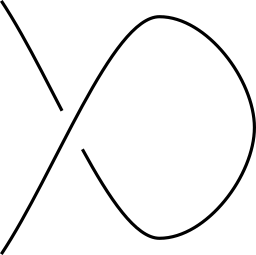}} =a\raisebox{-.5cm}{\includegraphics[scale=.6]{images/Kauf.png}}\,, \tag{F2} \\
 \raisebox{-.5cm}{\includegraphics[scale=.6]{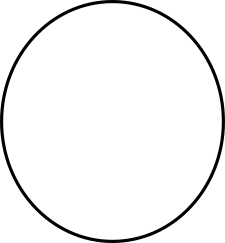}} \;\sqcup L =\left(\frac{a-a^{-1}}{z}+1\right)L, \tag{F3}
\end{gather*}
where $F_L:=1$ if $L$ is the empty link.\footnote{This is equivalent to choosing the normalization $F_U=\frac{a-a^{-1}}{z}+1$, where $U$ is the $0$-framed unknot.}

The ordinary BMW algebra, $\operatorname{BMW}_n$ (see \cite{BW, Mu}) can be defined as the Kauffman skein module for framed tangles in $[0,1]\times \R^2 \cong J^1[0,1]$ with $n$ boundary points on each component. In more detail, let $\mathfrak{Fr}_n$ denote the set of isotopy classes of {\it framed $n$-tangles} where we require that near $\partial J^1[0,1]$ framed $n$-tangles agree with the horizontal lines $z=i$, $y =0$, $1\leq i \leq n$ with framing vector $\frac{\partial}{\partial y}$; this requirement is maintained during isotopies. Working over the field of rational functions $\mathbb{F} = \Z(a,s)$ with $z=s-s^{-1}$, $\operatorname{BMW}_n= \mathbb{F}\mathfrak{Fr}_n/\mathcal{T}$ where $\mathcal{T}$ is the $\mathbb{F}$-submodule generated by the Kauffman polynomial skein relations. Multiplication in $\operatorname{BMW}_n$ is as in the Legendrian case, i.e., the product $\alpha \cdot \beta$ is $\beta$ stacked to the left of~$\alpha$.

For viewing framed $n$-tangles diagrammatically, we continue to use the $xz$-plane for projections, and we require the framing to be globally given by $\frac{\partial}{\partial y}$, that is, perpendicular to the projection plane and away from the viewer. This framing becomes isotopic to the blackboard framing when $n$-tangles are closed to become links in $J^1S^1$ (by identifying the left and right boundaries of $J^1[0,1]$). To view a Legendrian $n$-tangle $L \subset J^1[0,1]$ as a framed $n$-tangle $\operatorname{fr}(L) \in \mathfrak{Fr}_n$ we form a diagram for $\operatorname{fr}(L)$ from the front diagram of $L$ by smoothing left cusps and adding a small loop with a negative crossing at right cusps. See Fig.~\ref{fig:framed}. This has the property that the closure of $\operatorname{fr}(L)$ in $J^1S^1$ is framed isotopic to the Legendrian closure of~$L$ with its contact framing.

\begin{figure}
 \centering
		\labellist
\small
\pinlabel $L$ [t] at 64 0
\pinlabel $\operatorname{fr}(L)$ [t] at 250 0
\endlabellist

 \includegraphics{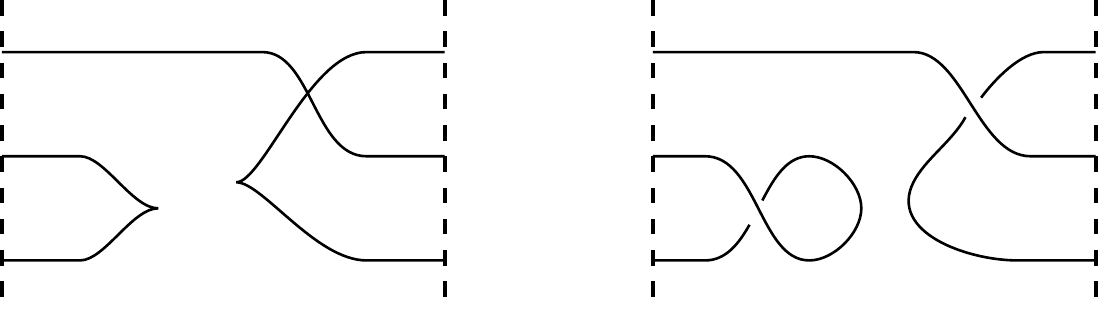}
		
		\quad
		
 \caption{The framed $n$-tangle $\operatorname{fr}(L)$ associated to a Legendrian $n$-tangle $L \subset J^1[0,1]$.}
 \label{fig:framed}
\end{figure}

The topological satellite operation produces from a framed knot $K \subset J^1\R$ and a framed $n$-tangle $L \subset J^1[0,1]$ a satellite link $S(K,L) \subset J^1\R$. A diagram for $S(K,L)$ arises from taking a blackboard framed diagram for $K$ and placing a blackboard framed diagram for the closure of $L$ in an (immersed) annular neighborhood of $K$ in the projection plane. When $K$ and $L$ are Legendrian, the framed knot type of $S(K,L)$ agrees with that of the previously defined Legendrian satellite (with contact framing).
 For $K \subset J^1\R$ (a framed link) and $\eta \in \operatorname{BMW}_n$, the Kauffman polynomial of the satellite $F_{S(K,\eta)}$ is defined by linearity as in (\ref{eq:BMWsatellite}).

\subsection[Symmetrizer in $\operatorname{BMW}_n$ and the $n$-colored Kauffman polynomial]{Symmetrizer in $\boldsymbol{\operatorname{BMW}_n}$ and the $\boldsymbol{n}$-colored Kauffman polynomial}

 The $n$-colored Kauffman polynomial is defined by satelliting with the {\it symmetrizer} $\mathcal{Y}_n \in \operatorname{BMW}_n$. (More general, colored Kauffman polynomials, where the coloring is by a partition $\lambda$, can be defined using other idempotents in $\operatorname{BMW}_n$ and are related to quantum invariants of type~$B$,~$C$, and~$D$; see, e.g.,~\cite{BlBe}.) The symmetrizer $\mathcal{Y}_n$ is characterized as the unique non-zero element of $\operatorname{BMW}_n$ that is idempotent, i.e., has $\mathcal{Y}_n^2 = \mathcal{Y}_n$, and satisfies
\begin{gather}
 \mathcal{Y}_n\sigma_i =s\mathcal{Y}_n, \qquad \text{for }1\leq i <n. \quad \quad \quad (\mbox{crossing absorbing property})\label{eats}
\end{gather}
The following inductive formula for $\mathcal{Y}_n$ is due to Heckenberger and Sch{\"u}ler in \cite{Heck}. Consider $Y_n \in \operatorname{BMW}_n$ defined by
 \begin{gather}
Y_1 = 1, \notag\\
\label{inductiveform}
 Y_n=Y_{n-1}\gamma_n+Y_{n-1}\frac{z}{1-s^{2n-3}a} \left(\sum_{i=1}^{n-1}as^{2n-2i-1}\left(\sum_{j=0}^{i-1}s^{j}D_{i-j,i+1}\right)\right),
\end{gather}
where
 \[D_{i,j}:=\begin{cases} {\includegraphics[scale=.8]{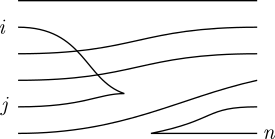}}\,, &\raisebox{.7cm}{ $i<j$,}\\
0,&\;i\geq j.
\end{cases}
\] The symmetrizer is then obtained as the normalization
$\mathcal{Y}_n=\frac{Y_n}{c_n}$ (with $c_n$ as in Definition~\ref{def:nruling}). For example, $\mathcal{Y}_2=\frac{1}{c_2}\big({\rm id}+s\sigma_1+\frac{sza}{1-sa}e_1\big)$. Note that $Y_n$ is quasi-idempotent (i.e., $Y_n^2 = c_n Y_n$) and also has the crossing absorbing property~(\ref{eats}).

\begin{Remark} \looseness=-1 Our notations $\sigma_i$, $s$, $a$, $e_i$ translate into the notations from \cite{Heck} as $g_i$, $q$, $r$, $r^{-1} e_i$. (The~$r^{-1}$ is because our $e_i$ is a Legendrian hook, so that $\operatorname{fr}(e_i)$ has an extra loop that produces the $r^{-1}$ factor.) Our $D_{i,j}$ is such that $a \cdot D_{i-j, i+1}$ is the $j$-th term in $d^+_{n,i}$ from~\cite{Heck}, the factor of $a$ again arising from our use of Legendrian tangles in representing the quasi-idempotent~$Y_n$. Note that \cite[Proposition~1]{Heck} gives an inductive formula for $\mathcal{Y}_n$ (notated there as~$S_n$) rather than~$Y_n$, and the denominator $q^n [\![ n ]\!]$ that appears there is accounted for by our factor of $\frac{1}{c_n}$ relating~$Y_n$ and~$\mathcal{Y}_n$.
\end{Remark}

\begin{Definition} \label{def:Kauffman}
Given a framed knot $K\subset\R^3$, we define the {\it $n$-colored Kauffman polynomial} of~$K$ by
\[F_{n,K}(a,s)=F_{S(K,\mathcal{Y}_n)}(a,z)\vert_{z=s-s^{-1}},\]
 where $\mathcal{Y}_n$ is the symmetrizer in $\operatorname{BMW}_n$.
\end{Definition}

\subsection{Ruling polynomials via specializations of the Kauffman polynomial}
For a Legendrian link in $L \subset J^1\R$ the Thurston--Bennequin number satisfies the inequality $\operatorname{tb}(L) \leq -\deg_a \widehat{F}_L(a,z)$, where $\widehat{F}_L = a^{-w(L)} F_L$ is the framing independent version of the Kauffman polynomial obtained by normalizing $F_L$ using the writhe of a diagram for~$L$. (See~\cite{FT, Ng3, T}.) This is equivalent to the inequality
\begin{gather} \label{eq:Estimate}
\deg_a F_L(a,z) \leq 0.
\end{gather}
 Thus, when $L$ is Legendrian $F_L \in \Z\big[a^{\pm 1}, z^{\pm}\big]$ does not contain positive powers of~$a$, and a~specialization $F_L(a,z)|_{a^{-1}=0}$ arises from simply setting $a^{-1}=0$.

\begin{Theorem}[\cite{R1}]\label{thm:rulings and the Kauffman poly}
For any Legendrian link $L\subset J^1\R$, the ungraded ruling polynomial is the specialization \[R^1_L(z)=F_L(a,z)|_{a^{-1}=0}.\]
\end{Theorem}
We will want to perform a similar specialization on elements of $\operatorname{BMW}_n$ whose coefficients in $\mathbb{F} = \Z(a,s)$ may not be Laurent polynomials in $a$.
 To do so, note that the notion of degree in $a$ extends to arbitrary non-zero rational functions via
 $\deg_a\colon \Z(a,s) \setminus \{0\} \rightarrow \Z$,
\[
 \deg_a\left(\frac{f}{g}\right) = \deg_a f - \deg_a g, \qquad \mbox{for $f,g \in \Z[a,s]$}.
\]
In addition, we use the convention that $\deg_a0:=-\infty$. Moreover, on the subring $\mathbb{F}^- := \{F\in \mathbb{F} \,|  \allowbreak \deg_a F\leq 0 \}$ we can define a specialization $\vert_{a^{-1}=0}\colon \mathbb{F}^-\rightarrow \mathbb{Z}(s)$, by
\[
\frac{f}{g}\bigg\vert_{a^{-1}=0}:=\begin{cases}0, & \deg_a(f/g)<0,\\
\dfrac{c_a(f)}{c_a(g)},&\deg_a(f/g)= 0,\end{cases}
\]
where $c_a(f) \in \Z[s]$ denotes the leading coefficient in $a$ of $f$.
\begin{Proposition} The specialization $\vert_{a^{-1}=0}\colon \mathbb{F}^-\rightarrow \mathbb{Z}(s)$ is a well-defined, unital ring homomorphism.
 \end{Proposition}
\begin{proof}Suppose $\frac{f_1}{g_1}=\frac{f_2}{g_2}$ in $\F^-$. Then $f_1g_2=g_1f_2$ and so $c_a(f_1)c_a(g_2)=c_a(g_1)c_a(f_2)$, since $c_a(fg)=c_a(f)c_a(g)$ is true for polynomials $f$ and $g$. It follows that $\frac{f_1}{g_1}\big\vert_{a^{-1}=0}=\frac{f_2}{g_2}\big\vert_{a^{-1}=0}$ and so the specialization $\big\vert_{a^{-1}=0}$ is well-defined. Note that $1\vert_{a^{-1}=0}=1$ follows from the definition. Now suppose $\frac{f_1}{g_1},\frac{f_2}{g_2}\in\F^-$. If $\deg_a({f_1}/{g_1})=\deg_a({f_2}/{g_2})=0$, then $\deg_a(f_1g_2)=\deg_a(g_1f_2)=\deg_a(g_1g_2)$. In which case, \begin{align*}
\left(\frac{f_1}{g_2}+\frac{f_2}{g_2}\right)\bigg\vert_{a^{-1}=0} & =\frac{f_1g_2+g_1f_2}{g_1g_2}\bigg\vert_{a^{-1}=0}\\
& =\frac{c_a(f_1)c_a(g_2)+c_a(g_1)c_a(f_2)}{c_a(g_1)c_a(g_2)} =\frac{f_1}{g_1}\bigg\vert_{a^{-1}=0}+\frac{f_2}{g_2}\bigg\vert_{a^{-1}=0}.
\end{align*} If $\deg_a({f_1}/{g_1}),  \deg_a({f_2}/{g_2})<0$, then
\[
\deg_a(f_1g_2+g_1f_2)\leq \max\{\deg_a(f_1g_2),\deg_a(g_1f_2)\}<\deg_a(g_1g_2)
\] and so $\big(\frac{f_1}{g_1}+\frac{f_2}{g_2}\big)\big\vert_{a^{-1}=0}= 0 = \frac{f_1}{g_1}\big\vert_{a^{-1}=0}+\frac{f_2}{g_2}\big\vert_{a^{-1}=0}$ holds by definition. Lastly, if $\deg_a(f_1/g_1)<\deg_a(f_2/g_2)=0$, then $\deg_a(f_1g_2+g_1f_2)=\deg_a(g_1f_2)=\deg(g_1g_2)$ and \[\left(\frac{f_1}{g_2}+\frac{f_2}{g_2}\right)\bigg\vert_{a^{-1}=0}=\frac{f_1g_2+g_1f_2}{g_1g_2}\bigg\vert_{a^{-1}=0}=\frac{c_a(g_1)c_a(f_2)}{c_a(g_1)c_a(g_2)}=\frac{f_1}{g_1}\bigg\vert_{a^{-1}=0}+\frac{f_2}{g_2}\bigg\vert_{a^{-1}=0}.\] It remains to show $\vert_{a^{-1}=0}$ preserves multiplication. In the case where
 $\deg_a(f_1/g_1)<0$, it follows that $\deg_a(f_1f_2)<\deg_a(g_1g_2)$ and so
 \[\left(\frac{f_1f_2}{g_1g_2}\right)\bigg\vert_{a^{-1}=0}=0 =\frac{f_1}{g_1}\bigg\vert_{a^{-1}=0}\cdot\frac{f_2}{g_2}\bigg\vert_{a^{-1}=0}.\]
 The remaining case where $\deg_a(f_1/g_1)=\deg_a(f_1/g_1)=0$ immediately follows from the fact that $c_a$ preserves multiplication for then
\begin{gather*}
\left(\frac{f_1f_2}{g_1g_2}\right)\bigg\vert_{a^{-1}=0}=\frac{c_a(f_1f_2)}{c_a(g_1g_2)} =\frac{c_a(f_1)c_a(f_2)}{c_a(g_1)c_a(g_2)} =\frac{f_1}{g_1}\bigg\vert_{a^{-1}=0}\cdot\frac{f_2}{g_2}\bigg\vert_{a^{-1}=0}.\tag*{\qed}
\end{gather*}\renewcommand{\qed}{}
\end{proof}

In the following we will make a connection between the Legendrian BMW algebra from Section~\ref{sec:ncolored} and $\operatorname{BMW}_n$. Towards this end, define
\[
\operatorname{BMW}^-_n= \operatorname{Span}_{\mathbb{F}^-}\mathfrak{Leg}_n \subset \operatorname{BMW}_n
\]
to be the $\mathbb{F}^-$-submodule generated by (framed isotopy classes of) Legendrian $n$-tangles. Next, let $\mathfrak{m} \subset \mathbb{F}^-$ denote the maximal ideal $\mathfrak{m} = \{F \in \mathbb{F}^-\,|\, \deg_a F < 0\}$, that is the kernel of $\vert_{a^{-1}=0}$, and consider the quotient $\mathbb{F}^-$-module
\[
\operatorname{BMW}^\infty_n = \operatorname{BMW}^-_n/\mathfrak{m} \cdot \operatorname{BMW}^-_n,
\]
with the projection map notated as
\[
\operatorname{BMW}^-_n \rightarrow \operatorname{BMW}^\infty_n, \qquad y \mapsto y\vert_{a^{-1}=0}.
\]

Note that the estimate~(\ref{eq:Estimate}) shows that for any fixed Legendrian knot $K \subset J^1\R$ the $\mathbb{F}$-module homomorphism $\operatorname{BMW}_n \rightarrow \mathbb{F}$, $L \mapsto F_{S(K,L)}(a,s)$ maps $\operatorname{BMW}_n^-$ to $\mathbb{F}^-$. Moreover, the specialization $\operatorname{BMW}^-_n \rightarrow \mathbb{Z}(s)$, $L \mapsto F_{S(K,L)}(a,s)\big\vert_{a^{-1}=0}$ induces a well defined map \mbox{$\operatorname{BMW}^\infty_n \rightarrow \mathbb{Z}(s)$}.
\begin{Proposition} \label{prop:algebra}
There is an $\mathcal{R}$-algebra homomorphism $\varphi\colon \operatorname{BMW}_n^{\leg} \rightarrow \operatorname{BMW}_{n}^\infty$ induced by the map $\mathfrak{Leg}_n \rightarrow \mathfrak{Fr}_n$. Moreover, we have a commutative diagram
\[
\xymatrix{ \operatorname{BMW}^-_n \ar[rr]^{F_{S(K,\cdot)}} \ar[d]_{|_{a^{-1}=0}} & & \mathbb{F}^- \ar[rr]^{|_{a^{-1}=0}} & & \mathbb{Z}(s). \\ \operatorname{BMW}^\infty_n \ar[rrrru] & & & & \\ \operatorname{BMW}^\leg_n \ar[u]^{\varphi} \ar[rrrruu]_{R^1_{S(K,\cdot)}} & & & &}
\]
\end{Proposition}
\begin{proof} To see that $\varphi$ is well-defined, note that the ruling polynomial relation (R1) holds in $\operatorname{BMW}_n$ since it is implied by the Kauffman relation (F1). Moreover, using the (F2), we see that when $L \subset J^1[0,1]$ has a zig-zag, and $L' \subset J^1[0,1]$ is the Legendrian obtained by removing the zig-zag from $L$, we have $L = a^{-1}L'$ in $\operatorname{BMW}^-$, and this implies that $[L]=0$ in $\operatorname{BMW}^\infty$ as required by (R2). Finally, to verify (R3) when $K \in \leg_n$ we can compute in $\operatorname{BMW}^\infty_n$
\begin{align*}
K \sqcup \raisebox{-.3cm}{\includegraphics[scale=.5]{images/unknot}} & = a^{-1} K\sqcup \raisebox{-.3cm}{\includegraphics[scale=.4]{images/Kauf7}} = a^{-1}\left(\frac{a-a^{-1}}{z} +1\right) K \\
& = z^{-1}K + \big(a^{-1}-a^{-2}/z\big) K = z^{-1}K,
\end{align*}
since $\big(a^{-1}-a^{-2}/z\big) \in \mathfrak{m}$.

The upper triangle of the diagram is commutative by definitions, and the commutativity of the lower triangle follows from Theorem~\ref{thm:rulings and the Kauffman poly}.
\end{proof}

\begin{Conjecture}When $\operatorname{BMW}^\leg_n$ is defined over $\mathbb{Z}(s)$, the map $\varphi$ is an algebra isomorphism.
\end{Conjecture}

\begin{Remark} A similar conjecture can be made involving the (suitably defined) $1$-graded ruling polynomial skein module and the Kauffman skein module of any contact $3$-manifold, $M$.
 \end{Remark}

Recall the element $L_n \in \operatorname{BMW}^\leg_n$ from Section~\ref{sec:Ln}.
\begin{Proposition} \label{prop:BMW}
For any $n\geq 1$, we have
\begin{enumerate}\itemsep=0pt
\item[$1)$] $\mathcal{Y}_n \in \operatorname{BMW}^-_n$, and
\item[$2)$] $\varphi\big(\frac{1}{c_n} L_n\big) = \mathcal{Y}_n\big\vert_{a^{-1}=0}$ holds in $\operatorname{BMW}^\infty_n$.
\end{enumerate}
\end{Proposition}
We prove (1) now; the proof of~(2) is deferred until Section~\ref{sec:prop}.
\begin{proof}[Proof of (1)] Note that all the tangles involved in the inductive characterization of $Y_n$ from (\ref{inductiveform}) are Legendrian with coefficients in~$\mathbb{F}^-$, and the normalizing factor $\frac{1}{c_n}$ also belongs to~$\mathbb{F}^-$.
\end{proof}

The following is the second equality in the statement of Theorem~\ref{thm:main} from the introduction.
\begin{Theorem} \label{thm:main2}
For any Legendrian knot $K \subset J^1\R$, the $n$-colored Kauffman polynomial has $F_{n,K} \in \mathbb{F}^-$ and satisfies
\[
F_{n,K}|_{a^{-1}=0} = R^1_{n,K}.
\]
\end{Theorem}
\begin{proof}Since $\mathcal{Y}_n \in \operatorname{BMW}^-_n$ we get that $F_{n,K} = F_{S(K,\mathcal{Y}_n)} \in \mathbb{F}^-$, and since $\varphi\big(\frac{1}{c_n}L_n\big) = [\mathcal{Y}_n]$ the commutativity of the diagram in Proposition~\ref{prop:algebra} together with Proposition~\ref{thm:main1}
 shows that
\begin{gather*}
F_{n,K}|_{a^{-1}=0} = F_{S(K,\mathcal{Y}_n)}|_{a^{-1}=0} = R^1_{S(K, \frac{1}{c_n} L_n)} = R^1_{n,K}.\tag*{\qed}
\end{gather*}\renewcommand{\qed}{}
\end{proof}

\begin{Corollary}For each $n\geq2$, $\varphi\big(\frac{1}{c_n}L_n\big)$ is a central idempotent in $\operatorname{BMW}_n^\infty$.
\end{Corollary}
\begin{proof}This follows from Proposition~\ref{prop:BMW} since $\mathcal{Y}_n$ has this property already in $\operatorname{BMW}_n$. (See, e.g.,~\cite{BlBe, Heck}.)
\end{proof}

\subsection{Establishing (2) of Proposition \ref{prop:BMW}} \label{sec:prop}

We now embark on showing $\varphi\big(\frac{1}{c_n}L_n\big)=\mathcal{Y}_n\vert_{a^{-1}=0}$. Throughout this section we will work in $\operatorname{BMW}^{\infty}_n$, but we will simplify notation by writing $L_n$ for $\varphi(L_n)$. We begin with some preparatory lemmas that provide formulas for $L_n$ that are closer to the inductive formula for $Y_n$ from (\ref{inductiveform}). Lemma~\ref{lemma:beta} uses the hypothesis that $L_{n-1}$ has the crossing absorbing property. This assumption is later verified to be true (see Proposition~\ref{prop:almostmainthm}).

\begin{Lemma}\label{lemma:beta}Let $n\geq2$ and assume that, in $\operatorname{BMW}^\infty_n$, $L_{n-1}$ has the crossing absorbing pro\-per\-ty~\eqref{eats}. Then, in $\operatorname{BMW}^\infty_n$ we have
 \begin{equation*}
 {L_{n-1}\beta_n}=L_{n-1}\left(1-z\sum_{k=2}^n s^{2-k}\left[\sum_{j=1}^{k-2}-zs^{2+j-k} D_{n-j,n}+ D_{n-k+1,n}\right]\right)\gamma_n.
\end{equation*}
\end{Lemma}

\begin{proof}We use Lemma \ref{lemma:alpha}. For each $2\leq k\leq n$, let $\operatorname{cr}_\ell(C_{k,n})$ (respectively $\operatorname{cr}_r(C_{k,n})$) denote the set of crossings appearing in the left (resp.\ right) half of $C_{k,n}$, and abbreviate $C_k:=C_{k,n}$. Then,
\begin{align*}
 {L_{n-1}\beta_n}&=L_{n-1}\left(1-z\sum_{k=2}^n \sum_{X\subset \operatorname{cr}(C_k)}(-z)^{|X|}r_X(C_k)\right)\gamma_n\\
 &=L_{n-1}\left(1-z\sum_{k=2}^n \sum_{X\subset \operatorname{cr}_\ell(C_k)}\sum_{Y\subset \operatorname{cr}_r(C_k)}(-z)^{|X|+|Y|}r_{X\bigsqcup Y}(C_k)\right)\gamma_n\\
 &=L_{n-1}\left(1-z\sum_{k=2}^n \sum_{X\subset \operatorname{cr}_\ell(C_k)}(-z)^{|X|}\sum_{Y\subset \operatorname{cr}_r(C_k)}(-z)^{|Y|}r_{X\bigsqcup Y}(C_k)\right)\gamma_n \, .
\end{align*}
Since $L_{n-1}$ absorbs crossings $L_{n-1}r_{X \bigsqcup Y}(C_k)=L_{n-1}s^{k-2-|Y|}r_X(D_{n-k+1,n})$; see Fig.~\ref{fig:Absorbing}. Furthermore, because there are $\binom{k-2}{|Y|}$ subsets of $\operatorname{cr}_r(C_k)$ having $|Y|$ crossings, summing over $|Y|$ one obtains
\begin{gather*}
 L_{n-1}\sum_{Y\subset \operatorname{cr}_r(C_k)}(-z)^{|Y|}r_{X\bigsqcup Y}(C_k) =L_{n-1}\sum_{j=0}^{k-2}\binom{k-2}{j}(-z)^j s^{k-2-j}r_X(D_{n-k+1,n})\\
\qquad{} =L_{n-1}(s-z)^{k-2}r_X(D_{n-k+1,n})
=L_{n-1}s^{2-k}r_X(D_{n-k+1,n}).
\end{gather*}
Therefore, \[L_{n-1}\beta_n=L_{n-1}\left(1-z\sum_{k=2}^ns^{2-k}\sum_{X\subset \operatorname{cr}_\ell(C_k)}(-z)^{|X|}r_X(D_{n-k+1,n})\right)\gamma_n.\]

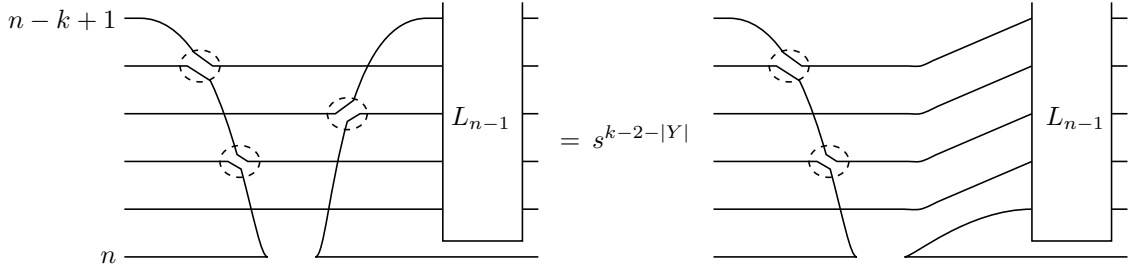
\begin{figure}
 \centering
		\labellist
 \small
\pinlabel $n$ [r] at 0 0
\pinlabel $n-k+1$ [r] at 0 112
\pinlabel $L_{n-1}$ [b] at 168 56
\pinlabel $L_{n-1}$ [b] at 448 56
\pinlabel $=\,s^{k-2-|Y|}$ [b] at 236 48
\endlabellist

 \quad \includegraphics[scale=.8]{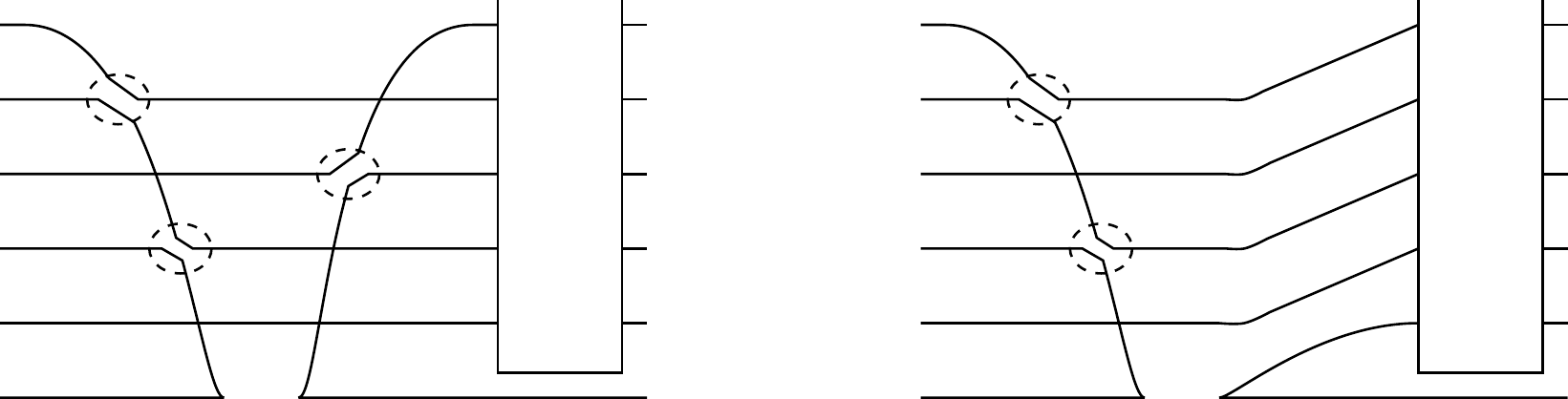}
		
 \caption{An illustration of the identity $L_{n-1}r_{X \bigsqcup Y}(C_k)=s^{k-2-|Y|}L_{n-1}r_X(D_{n-k+1,n})$, with the resolved crossings from $X \bigsqcup Y$ indicated in dotted ovals. The number of crossings in the right half of $r_{X\sqcup Y}(C_{k,n})$ is $k-2-|Y|$.} \label{fig:Absorbing}
\end{figure}

It remains to establish the innermost sum satisfies
\begin{gather} \label{eq:Xsum}
L_{n-1}\sum_{X\subset \operatorname{cr}_\ell(C_k)}(-z)^{|X|}r_X(D_{n-k+1,n}) = L_{n-1}\left[\sum_{j=1}^{k-2}-zs^{2+j-k} D_{n-j,n}+ D_{n-k+1,n}\right].
\end{gather}
When we perform the resolution by subsets of $\operatorname{cr}_\ell(C_k)$ we will not be able to feed all of the remaining crossings into $L_{n-1}$ (in fact, when $X$ is the empty set there are no crossings that we can push into $L_{n-1}$). We remedy this by partitioning the nonempty subsets of $\operatorname{cr}_\ell(C_k)$. Label the crossings in $\operatorname{cr}_\ell(C_k)$ by ascending $z$-coordinate as $c_1, \ldots, c_{k-2}$, and define $\chi_j=\{X\subset \operatorname{cr}_\ell(C_k)\colon \allowbreak j=\min\{i\,|\,c_i \in X\}\}$, for $1\leq j\leq k-2$. Let $X\in\chi_j$ be given. By definition $c_j\in X$ is the lowest crossing that is resolved in $r_X(D_{n-k+1,n})$. Isotopy allows us to push the $k-2-j-(|X|-1)$ remaining crossings lying above $c_j$ into $L_{n-1}$. Hence, using $i=|X|-1$ and that the requirement that $X \in \chi_j$ leaves $i$ choices from the $k-2-j$ crossings above $c_j$ to determine $X$, we have
 \begin{align*}
 L_{n-1}\sum_{X\in\chi_j}(-z)^{|X|}r_X(D_{n-k+1,n})&=L_{n-1}\sum_{i=0}^{k-2-j}\left(\begin{array}{c} k-2-j \\ i \end{array} \right)(-z)^{i+1}s^{(k-2-j)-i}D_{n-j,n} \\
		& = -zL_{n-1}(-z+s)^{k-2-j} D_{n-j,n} \quad \mbox{(binomial theorem)} \\
		 &=-zL_{n-1}s^{2+j-k}D_{n-j,n}.
 \end{align*}
Summing over all $j$, and adding the term $D_{n-k+1,n}$ for the case when $X$ is empty, establishes~(\ref{eq:Xsum}) and completes the proof.
\end{proof}
\begin{Lemma}\label{lemma:mult}For any $1\leq i\leq n$, in $\operatorname{BMW}^\infty_n$ we have
$D_{i,n}\gamma_n=\sum\limits_{r=0}^{n-i-1}s^rD_{i,n-r}$.
\end{Lemma}
\begin{proof}
This is just applying type II Reidemeister moves to see that
\[
D_{i,n}s^r\sigma_{n-1}\sigma_{n-2}\cdots\sigma_{n-r} =  \begin{cases} s^r D_{i,n-r}, &   r < n-i, \\ 0, & r \geq n-i,\end{cases}
\]
 since the fishtail relation (\ref{eq:fish}) may be applied when $r \geq n-i$. See Fig.~\ref{fig:DIJprod}.
\end{proof}

\begin{figure}[h]
 \centering
 \raisebox{-1.25cm}{\includegraphics[scale=.9]{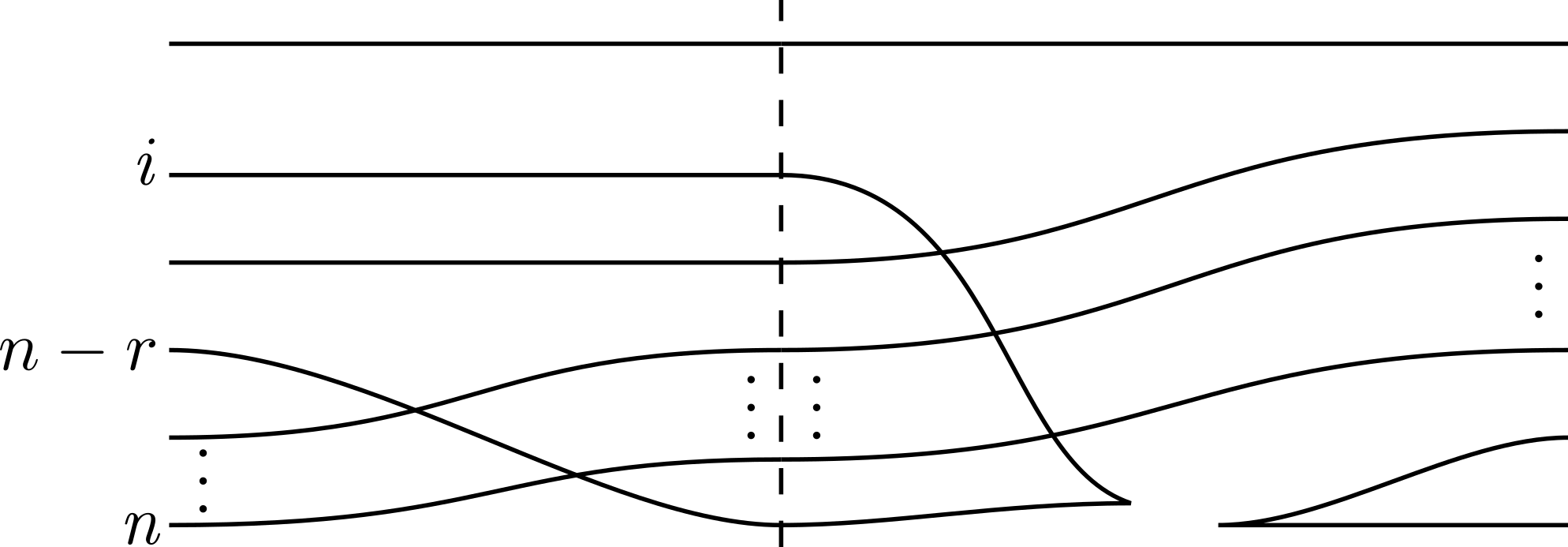}}\quad=\quad\raisebox{-1.25cm}{\includegraphics{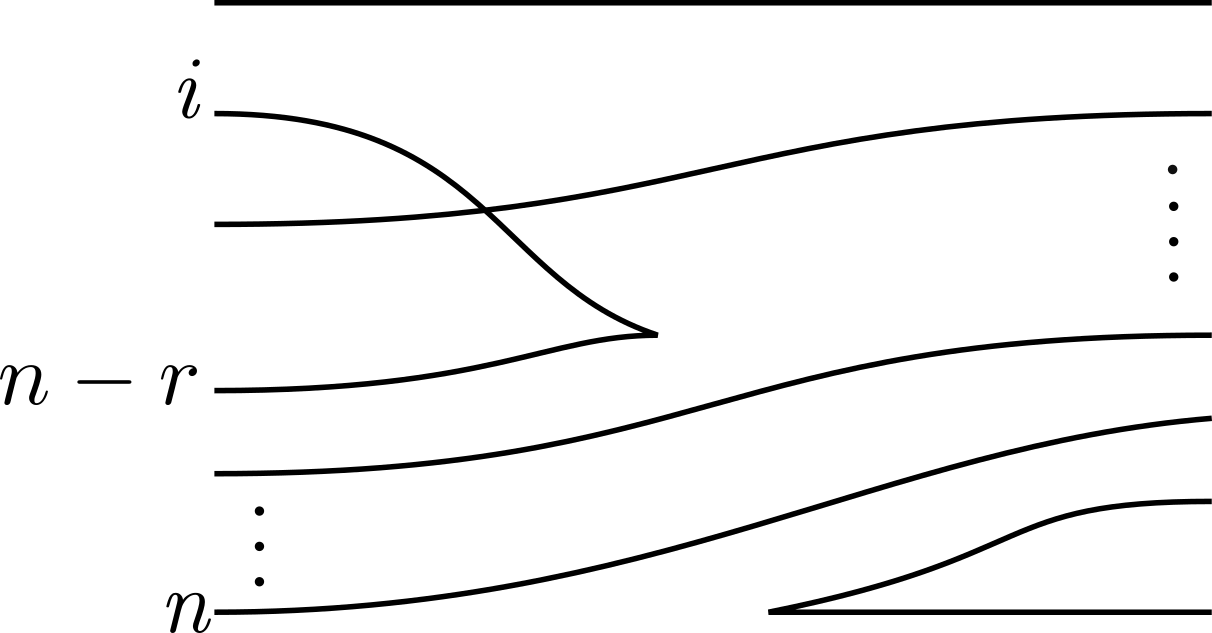}}

 \caption{An Illustration of the Type II Reidemeister moves used in Lemma \ref{lemma:mult}.}
		 \label{fig:DIJprod}
\end{figure}
Proposition \ref{prop:BMW} (2) follows from the following.
\begin{Proposition} \label{prop:almostmainthm}
For all $n\geq1$, $L_n$ has the following properties in $\operatorname{BMW}^\infty_n$:
\begin{itemize}\itemsep=0pt
 \item[$(i)$] $L_n=Y_n\vert_{a^{-1}=0}$,
 \item[$(ii)$] $L_n$ has the crossing absorbing property~\eqref{eats}.
\end{itemize}
\end{Proposition}
\begin{proof} The proof is by induction on $n$. In the base case $n=1$, (i) is immediate from definitions and (ii) is vacuous. Assume the result for $n-1$. Note that it suffices to establish (i) since the fact that $Y_n$ has the crossing absorbing property and $\vert_{a^{-1}=0}$ is a $\mathbb{Z}(s)$-algebra homomorphism would then allow us to verify (ii) via
\[
L_n\sigma_i=(Y_n\vert_{a^{-1}=0})\sigma_i=(Y_n\sigma_i)\vert_{a^{-1}=0}=(sY_n)\vert_{a^{-1}=0}=sL_n.
\]
 Showing $L_n=Y_n\vert_{a^{-1}=0}$ is the following a computation (with Lemmas \ref{lemma:beta} and \ref{lemma:mult} used at the 2nd and 3rd equality):
\begin{align*}
 L_n&=L_{n-1}\beta_n =L_{n-1}\left(1-z\sum_{k=2}^n s^{2-k}\left[\sum_{j=1}^{k-2}-zs^{2+j-k} D_{n-j,n}+ D_{n-k+1,n}\right]\gamma_n\right)\\
 &=L_{n-1}\left(\gamma_n-z\sum_{k=2}^n s^{2-k}\left[\sum_{j=1}^{k-2}-zs^{2+j-k}\sum_{r=0}^{j-1}s^rD_{n-j,n-r} + \sum_{r=0}^{k-2}s^rD_{n-k+1,n-r}\right]\right)\\
 &=L_{n-1}\left(\gamma_n-z\sum_{k=2}^n s^{2-k}\left[\sum_{r=0}^{k-3}\sum_{j=r+1}^{k-2}-zs^{2+j-k+r}D_{n-j,n-r} + \sum_{r=0}^{k-2}s^rD_{n-k+1,n-r}\right]\right)\\
 &=L_{n-1}\left(\gamma_n-z\left[\sum_{k=2}^n \sum_{r=0}^{k-3}\sum_{j=r+1}^{k-2}-zs^{4+j-2k+r}D_{n-j,n-r} + \sum_{k=2}^n\sum_{r=0}^{k-2}s^{2-k+r}D_{n-k+1,n-r}\right]\right)\\
	 &=L_{n-1}\left(\gamma_n-z\left[\sum_{r=0}^{n-3} \sum_{k=r+3}^{n}\sum_{j=r+1}^{k-2}-zs^{4+j-2k+r}D_{n-j,n-r} \right.\right.\\
& \left.\left. \quad{} + \sum_{r=0}^{n-2}\sum_{k=r+2}^{n}s^{2-k+r}D_{n-k+1,n-r}\right]\right)\\
 &=L_{n-1}\left(\gamma_n-zD_{1,2}-z\sum_{r=0}^{n-3}s^r\left[ \sum_{k=r+3}^{n}\sum_{j=r+1}^{k-2}-zs^{4+j-2k}D_{n-j,n-r}\right.\right.\\
& \left.\left. \quad{} +\sum_{j=r+1}^{n-1}s^{2-j-1}D_{n-j,n-r}\right]\right)\\
 &=L_{n-1}\left(\gamma_n-zD_{1,2}-z\sum_{r=0}^{n-3}s^r\left[ \sum_{j=r+1}^{n-2}s^j D_{n-j,n-r}\left(\sum_{k=j+2}^{n}-zs^{4-2k}\right)\right.\right.\\
& \left.\left. \quad{}
 +\sum_{j=r+1}^{n-1}s^{2-j-1} D_{n-j,n-r}\right]\right)\\
 &=L_{n-1}\left(\gamma_n-z D_{1,2}-z\sum_{r=0}^{n-3}s^r\left[ \sum_{j=r+1}^{n-2}s^j D_{n-j,n-r}\left(z\frac{s^{4-2n}-s^{2-2j}}{s^2-1}\right)\right.\right.\\
& \left.\left. \quad{} +\sum_{j=r+1}^{n-1}s^{2-j-1} D_{n-j,n-r}\right]\right)\\
 &=L_{n-1}\left(\gamma_n-z D_{1,2}-z\sum_{r=0}^{n-3}s^r\left[ \sum_{j=r+1}^{n-2}\left(s^{j-2n+3}-s^{1-j}\right) D_{n-j,n-r}\right.\right.\\
& \left.\left. \quad{}+\sum_{j=r+1}^{n-1}s^{2-j-1} D_{n-j,n-r}\right]\right)\\
 &=L_{n-1}\left(\gamma_n-z D_{1,2}-z\sum_{r=0}^{n-3}s^r\Bigg[s^{2-n} D_{1,n-r}
  \right.\\
& \left.\left. \quad{} +\sum_{j=r+1}^{n-2}\left(s^{j-2n+3}-s^{1-j}+s^{2-j-1}\right) D_{n-j,n-r}\right]\right)\\
 &=L_{n-1}\left(\gamma_n-z D_{1,2}-z\sum_{r=0}^{n-3}s^r\left[s^{2-n} D_{1,n-r} +\sum_{j=r+1}^{n-2}\left(s^{j-2n+3}\right) D_{n-j,n-r}\right]\right)\\
 &=L_{n-1}\left(\gamma_n-z D_{1,2}-z\sum_{r=0}^{n-3}s^r\sum_{j=r+1}^{n-1}s^{j-2n+3} D_{n-j,n-r}\right)\\
 &=L_{n-1}\left(\gamma_n-\frac{z}{s^{2n-3}}\left(s^{2n-3} D_{1,2}+\sum_{r=0}^{n-3}s^r\sum_{j=r+1}^{n-1}s^{j} D_{n-j,n-r}\right)\right)\\
 &=L_{n-1}\left(\gamma_n-\frac{z}{s^{2n-3}}\left(s^{2n-3} D_{1,2}+\sum_{i=2}^{n-1}s^{n-i-1}\sum_{j=n-i}^{n-1}s^j D_{n-j,i+1}\right)\right)\\
 &=L_{n-1}\left(\gamma_n-\frac{z}{s^{2n-3}}\left(s^{2n-3} D_{1,2}+\sum_{i=2}^{n-1}s^{n-i-1}\sum_{j=0}^{i-1}s^{n-i+j} D_{i-j,i+1}\right)\right)\\
 &=L_{n-1}\left(\gamma_n-\frac{z}{s^{2n-3}}\left(s^{2n-3} D_{1,2}+\sum_{i=2}^{n-1}s^{2n-2i-1}\sum_{j=0}^{i-1}s^{j} D_{i-j,i+1}\right)\right)\\
 &=L_{n-1}\left(\gamma_n-\frac{z}{s^{2n-3}}\left(\sum_{i=1}^{n-1}s^{2n-2i-1}\sum_{j=0}^{i-1}s^{j} D_{i-j,i+1}\right)\right) \\
		& = Y_{n-1}\big\vert_{a^{-1}=0} \cdot \left(\gamma_n+\frac{z}{1-s^{2n-3}a}\left(\sum_{i=1}^{n-1}as^{2n-2i-1} \sum_{j=0}^{i-1}s^{j} D_{i-j,i+1}\right) \right)\bigg\vert_{a^{-1}=0} = Y_n\big\vert_{a^{-1}=0}.
\end{align*}
At the last equality, we used \eqref{inductiveform}.
\end{proof}

\section[The $n$-colored ruling polynomial and representation numbers]{The $\boldsymbol{n}$-colored ruling polynomial and representation numbers} \label{sec:5}

In this section, we show that the $1$-graded $n$-colored ruling polynomial agrees with the $1$-graded, total $n$-dimensional representation number defined in~\cite{LeRu}; see Theorem~\ref{thm:Rep}. After a brief review of Legendrian contact homology and relevant material from~\cite{LeRu}, the remainder of the section contains the proof of Theorem~\ref{thm:Rep}.

\subsection{Review of the Legendrian contact homology DGA} We assume familiarity with the Legendrian contact homology differential graded algebra, (abbrv. LCH DGA), aka.\ the Chekanov--Eliashberg algebra, in the setting of Legendrian links in $J^1M$ with $M = \R$ or $S^1$, and refer the reader to any of \cite{Chekanov02, EN, ENS, Ng1, NgR2012} for this background material. We continue to use coordinates $(x,y,z) \in J^1M = T^*M \times \R = M \times \R^2$, and to view projections to $S^1\times \R$ in $[0,1]\times \R$ with the left and right boundary identified.
 The Reeb vector field is $\frac{\partial}{\partial z}$, so Reeb chords of $K$ are in bijection with double points of the {\it Lagrangian projection} aka. the {\it $xy$-diagram} of $K$ which is the projection to $T^*M$.
 Representation numbers are defined in \cite{LeRu} using the fully non-commutative version of the LCH DGA associated to a Legendrian knot or link $K$ equipped with a collection of base points, $*_1, \ldots, *_\ell$, with the requirement that every component of $K$ has at least one base point. The resulting DGA, notated $(\mathcal{A}(K), \partial)$ or $(\mathcal{A}(K, *_1, \ldots, *_\ell), \partial)$ when the choice of base points should be emphasized, is an associative, non-commutative algebra with identity generated over $\Z$ by
\begin{itemize}\itemsep=0pt
\item[(i)] the Reeb chords of $K$, denoted $b_1, \ldots, b_r$, and
\item[(ii)] invertible generators $t_1^{\pm1}, \ldots, t_\ell^{\pm1}$ corresponding to the base points $*_1, \ldots, *_\ell$.
\end{itemize}
There are no relations other than $t_it_i^{-1} = t_i^{-1}t_i=1$. The differential $\partial$ vanishes on the $t_i$; for a Reeb chord, $a$, the differential $\partial a$ is defined via a signed count of rigid holomorphic disks in~$T^*M$ with boundary on the Lagrangian projection of $K$ and having a single positive boundary puncture at $a$ and an arbitrary number of negative boundary punctures. Each such a disk $u$ contributes a term $\pm w(u)$ to $\partial a$ where~$w(u)$ is the product of base point generators and negative punctures as they appear in counter-clockwise order along the boundary of the domain of~$u$, starting from the positive puncture at~$a$. Occurrences of~$t_i$ appear with exponent~$\pm1$ according to the oriented intersection number of $\partial u$ with $*_i$. See Fig.~\ref{fig:PointedDisk}. For associating $\pm1$ signs to disks, we use the conventions as in \cite{HenryRu, LeRu}. Most results of this section concern DGA representations defined over a field of characteristic~$2$, and in this case it suffices to work with the version of~$\mathcal{A}(K)$ defined over $\Z/2$ where the~$\pm1$ signs become irrelevant.

\begin{figure}
\labellist
\small
\pinlabel $a$ [l] at 186 102
\pinlabel $+$ [r] at 166 102
\pinlabel $b_1$ [b] at 136 180
\pinlabel $b_2$ [b] at 24 182
\pinlabel $b_3$ [tl] at 84 20
\pinlabel $-$ [t] at 132 152
\pinlabel $-$ [t] at 28 162
\pinlabel $-$ [b] at 86 34
\endlabellist
\centerline{\includegraphics[scale=.6]{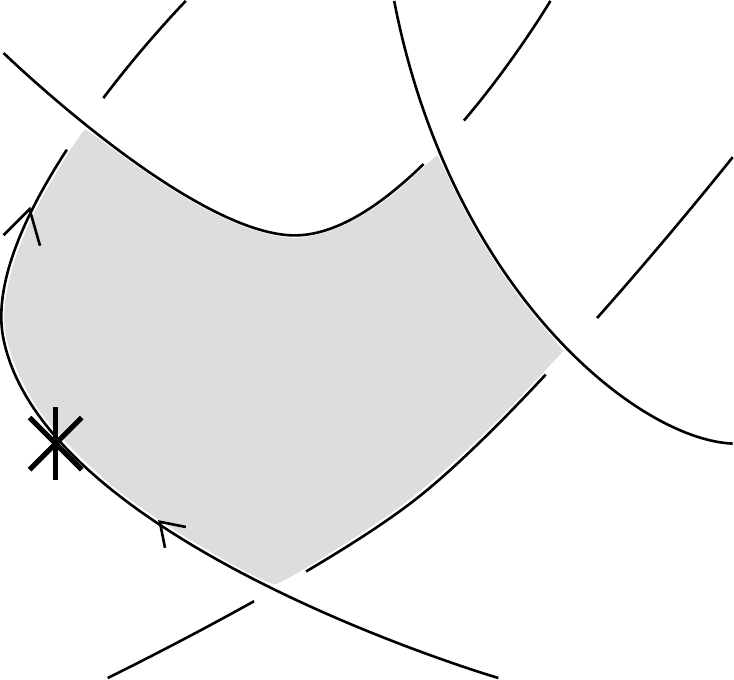}}

\caption{A holomorphic disk contributing the term $\partial a = \pm b_1 b_2 t^{-1} b_3 + \cdots$
 to the differential of $\mathcal{A}(K)$.}
\label{fig:PointedDisk}
\end{figure}
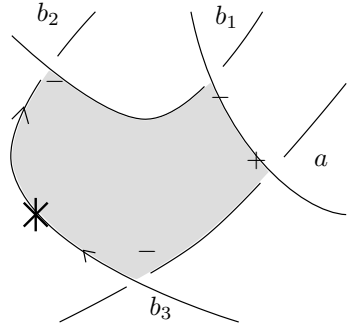

\subsection{1-graded representation numbers} In~\cite{LeRu}, Legendrian invariant $m$-graded representation numbers are defined for any non-negative integer $m\geq 0$ by considering DGA homomorphisms that are only required to preserve grading mod $m$. In the current article, we are concerned only with $1$-graded representations, which we will refer to as {\it ungraded} representations since the grading condition becomes vacuous when $m=1$. We review definitions from~\cite{LeRu} in the ungraded setting.

Let $V$ be a vector space over a field, $\mathbb{F}$, with $\operatorname{char}(\mathbb{F}) = 2$, and let $d\colon V \rightarrow V$ be an ungraded differential on~$V$, i.e., $d$ is just a linear map satisfying $d^2=0$.
 Then, $d$ induces a differential on the endomorphism algebra
\[
\delta\colon \ \operatorname{End}(V) \rightarrow \operatorname{End}(V), \qquad \delta(T) = d \circ T + T \circ d
\]
making $(\operatorname{End}(V),\delta)$ into an ungraded DGA, i.e.,
 $\delta$ satisfies $\delta^2=0$ and $\delta(T_1T_2) = \delta(T_1)T_2 + T_1\delta(T_2)$. An {\it ungraded representation} of a DGA, $(\mathcal{A}, \partial)$, on $(V,d)$ is an ungraded DGA homomorphism
\[
f\colon \ (\mathcal{A}, \partial) \rightarrow (\operatorname{End}(V), \delta),
\]
i.e., a ring homomorphism satisfying $f(1)=1$ and $f \circ \partial = \delta \circ f$.
In the $1$-dimensional case where $V = \mathbb{F}$ and $d=0$, an ungraded representation on $(\mathbb{F},0)$ is also called an {\it ungraded augmentation}.

For $\mathcal{A} = \mathcal{A}(K)$, we use the notation $\overline{\operatorname{Rep}}_1(K, (V, d))$ for the set of all ungraded representations of $(\mathcal{A},\partial)$ on $(V,d)$ and $\overline{\operatorname{Aug}}_1(K, \mathbb{F})$ for the set of augmentations to $\mathbb{F}$.
 In the case where $K$ is connected with base points $*_1, \ldots, *_\ell$ appearing in order starting at $*_1$ and following the orientation of $K$, given a subset $T \subset {\rm GL}(V)$ we will use the notation
 $\overline{\operatorname{Rep}}_1(K,(V,d),T)$ to denote the set of those $f \in \overline{\operatorname{Rep}}_1(K, (V, d))$ such that $f(t_1\cdots t_\ell) \in T$. In particular, $\overline{\operatorname{Rep}}_1(K, (V, d)) =\overline{\operatorname{Rep}}_1(K,(V,d), {\rm GL}(V))$.

\begin{Definition} \label{def:total} Let $\mathbb{F}_q$ denote the finite field of order $q$ with $q$ a power of $2$. The ($1$-graded) {\it total $n$-dimensional representation number} of $K$ is defined by
\begin{align*}
\operatorname{Rep}_1\big(K, \mathbb{F}^n_q\big) & := \big|\operatorname{End}\big(\mathbb{F}^n_q\big)\big|^{-{\rm rb}(K)/2}  |{\rm GL}(n, \mathbb{F}_q) |^{-\ell} \big| \overline{\operatorname{Rep}}_1\big(K, \big(\mathbb{F}^n_q, 0\big)\big) \big| \\
 & = \big(q^{n^2}\big)^{-{\rm rb}(K)/2}   \left(q^{n(n-1)/2}   \prod_{m=1}^n\big(q^m-1\big)\right)^{-\ell} \big| \overline{\operatorname{Rep}}_1\big(K, \big(\mathbb{F}^n_q, 0\big)\big) \big|,
\end{align*}
 where ${\rm rb}(K)$ is the number of Reeb chords of $K$ and $\ell$ is the number of basepoints.
\end{Definition}

That $\operatorname{Rep}_1\big(K, \mathbb{F}^n_q\big)$ only depends on the Legendrian isotopy type of $K$ is established in \cite[Proposition~3.10]{LeRu}. The following theorem shows that the $1$-graded $n$-colored ruling polynomial and the total $n$-dimensional representation numbers of~$K$ are equivalent Legendrian invariants.

\begin{Theorem} \label{thm:Rep}
Let $K \subset J^1 \R$ be a $($connected$)$ Legendrian knot and $\mathbb{F}_q$ a finite field of order $q$ with characteristic~$2$. Then,
\[
R^1_{n,K}(q) = \operatorname{Rep}_1(K, \mathbb{F}_q),
\]
where $R^1_{n,K}(q) = R^1_{n,K}(s)\big\vert_{s=q^{1/2}}$.
\end{Theorem}

The proof of Theorem \ref{thm:Rep} rests on the following relation between representation counts and reduced ruling polynomials which we will establish in Sections~\ref{sec:strat}--\ref{sec:augsat} using extensions of results from~\cite{HenryRu} and~\cite{LeRu}. In \cite[Section~4]{LeRu}, a ``path subset''\footnote{The {\it path subset} $B_\beta$ is the subset of ${\rm GL}(n,\mathbb{F}_q)$ arising from specializing the path matrix $P^{xy}_\beta$ using arbitrary ring homomorphisms from $\mathcal{A}(\beta)$ to $\mathbb{F}_q$. The {\it path matrix} $P^{xy}_\beta$ is a matrix whose entries belong to $\mathcal{A}(\beta)$ and record certain left-to-right paths through the $xy$-diagram of $\beta$ that reflect the possible behavior of boundaries of holomorphic disks bordering $\beta$ from above. See \cite[Section~4.1]{LeRu}. Note that~\cite{LeRu} uses the opposite convention for composing Legendrian $n$-tangles, with $\alpha \cdot \beta$ defined as $\alpha$ stacked to the {\it left} of~$\beta$. Consequently, for a permutation $\beta \in S_n$, the permutation braid $\beta \subset J^1S^1$ used in the present article corresponds to $\beta^{-1} \subset J^1S^1$ in~\cite{LeRu}. As a~result, the notations $B_\beta$ and $P^{xy}_\beta$ used here correspond to $B_{\beta^{-1}}$ and $P_{\beta^{-1}}^{xy}$ in~\cite{LeRu}.} $B_\beta \subset {\rm GL}(n, \mathbb{F}_q)$ is associated to a reduced\footnote{A positive permutation braid $\beta\subset J^1S^1$ is {\it reduced} if its front diagram corresponds to a reduced braid word, i.e., one where the product $\sigma_i\sigma_{i+1}\sigma_i$ does not appear for any~$i$.} positive permutation braid, $\beta \in S_n$.

\begin{Lemma} \label{lem:count}
Let $\beta \in S_n$ be an $n$-stranded reduced positive permutation braid, and suppose that the $($connected$)$ Legendrian knot $K \subset J^1\R$ has its front diagram in plat position\footnote{A front diagram is in {\it plat position} if all left cusps appear at a common $x$-coordinate at the far left of the diagram and all right cusps appear at a common $x$-coordinate at the far right of the diagram.} and is equipped with $\ell$ base points where~$\ell$ is the number of right cusps of~$K$. Then,
\begin{gather*}
\big|\overline{\operatorname{Rep}}_1\big(K, \big(\mathbb{F}_q^n,0\big), B_\beta\big)\big| = |{\rm GL}(n, \mathbb{F}_q)|^{\ell-1}  q^{n(n-1)/2}  (q-1)^n   q^{n^2   {\rm rb}(K)/2}   q^{\lambda(\beta)/2}\widetilde{R}_{S(K, \beta)}(z),
\end{gather*}
where ${\rm rb}(K)$ is the number of Reeb chords of $K$, $\lambda(\beta)$ is the length of $\beta$, i.e., the number of crossings of $\beta$, and $z=q^{1/2}-q^{-1/2}$.
\end{Lemma}

\begin{proof}[Proof of Theorem \ref{thm:Rep}] Since both sides of the equation are Legendrian isotopy invariant, we may assume $K$ is in plat position.
In \cite[Proposition 4.14]{LeRu}, it is shown that ${\rm GL}(n, \mathbb{F}_q) = \sqcup_{\beta \in S_n} B_\beta$. (Actually, this coincides with the Bruhat decomposition of ${\rm GL}(n, \mathbb{F}_q)$.) Thus,
\[
\overline{\operatorname{Rep}}_1\big(K, \big(\mathbb{F}_q^n,0\big), {\rm GL}(n, \mathbb{F}_q)\big) = \sqcup_{\beta \in S_n} \overline{\operatorname{Rep}}_1\big(K, \big(\mathbb{F}_q^n,0\big),B_\beta\big),
\]
 and using Lemma \ref{lem:count} and Definition \ref{def:nruling} we compute
\begin{align*}
\operatorname{Rep}_1\big(K, \mathbb{F}_q^n\big) & = \big(q^{n^2}\big)^{-{\rm rb}(K)/2}  |{\rm GL}(n)|^{-\ell}  \sum_{\beta \in S_n} \overline{\operatorname{Rep}}_1\big(K, \big(\mathbb{F}_q^n,0\big),B_\beta\big) \\
 & = |{\rm GL}(n)|^{-1}   q^{n(n-1)/2}(q-1)^n \sum_{\beta \in S_n}q^{\lambda(\beta)/2} \widetilde{R}_{S(K, \beta)}(z) \\
& = \left( q^{n(n-1)/2}  \prod_{m=1}^n\big(q^m-1\big)\right)^{-1}  q^{n(n-1)/2}(q-1)^n \sum_{\beta \in S_n} q^{\lambda(\beta)/2} \widetilde{R}_{S(K, \beta)}(z) \\
& = \left(\prod_{m=1}^n\left(\frac{q^m-1}{q-1} \right) \right)^{-1} \sum_{\beta \in S_n} q^{\lambda(\beta)/2} \widetilde{R}_{S(K, \beta)}(z) \\
& = \frac{1}{c_n} \sum_{\beta \in S_n} q^{\lambda(\beta)/2} \widetilde{R}_{S(K, \beta)}(z) = R^1_{n,K}(q).
\end{align*}
(The notation $c_n$ is as in Definition \ref{def:nruling} with $s = q^{1/2}$.)
\end{proof}

\subsection{Strategy of the proof of Lemma \ref{lem:count}} \label{sec:strat}

The relation between counts of representations on $\big(\mathbb{F}_q^n,0\big)$ and reduced ruling polynomials stated in Lemma \ref{lem:count} is based on refinements of two results:
\begin{enumerate}\itemsep=0pt
\item In \cite[Theorem~6.1]{LeRu}, it is shown that for any $n$-stranded reduced positive permutation braid $\beta \in S_n$,
 (using $1$ base point on $K$ and a particular $xy$-diagram of $S(K,\beta)$) when $\operatorname{char}(\mathbb{F}) = 2$ there is a bijection
\begin{gather} \label{eq:LeRu}
\bigsqcup_{d}\overline{\operatorname{Rep}}_1\big(K, \big(\mathbb{F}^n,d\big),B_\beta\big) \quad \longleftrightarrow \quad \overline{\operatorname{Aug}}_1(S(K, \beta), \mathbb{F}),
\end{gather}
where the union is over all strictly upper triangular differentials on $\mathbb{F}^n$.

\item In  \cite[Theorem 3.2]{HenryRu}, for any Legendrian $K' \subset J^1\R$ (using a particular $xy$-diagram of~$K'$) the set of augmentations of~$K'$ is decomposed into pieces
\begin{gather} \label{eq:HenryRu}
\operatorname{Aug}_1(K', \mathbb{F}) = \bigsqcup_{\rho} (\mathbb{F}^*)^{a(\rho)} \times \mathbb{F}^{b(\rho)},
\end{gather}
where the disjoint union is indexed by all normal rulings of $K'$ and the exponents~$a(\rho)$ and $b(\rho)$ are specified by the combinatorics of~$\rho$. Theorem~1.1 of~\cite{HenryRu} then applies the decomposition to relate the Legendrian invariant augmentation numbers with the ruling polynomial.
\end{enumerate}
The idea behind the proof of Lemma~\ref{lem:count} is then to check that the subset of $\operatorname{Aug}_1(S(K,\beta), \mathbb{F})$ corresponding under~(\ref{eq:LeRu}) to $\overline{\operatorname{Rep}}_1\big(K, \big(\mathbb{F}^n, 0\big), B_\beta\big)$, i.e., those representations with $d=0$, is the part of the disjoint union~(\ref{eq:HenryRu}) with $K' = S(K,\beta)$ that is indexed by {\it reduced} normal rulings of~$S(K, \beta)$. This is roughly what we shall do. However, complications arise as different (but Legendrian isotopic) $xy$-diagrams for the satellite $S(K,\beta)$ having different (but stable tame isomorphic) DGAs are used in~(\ref{eq:LeRu}) and~(\ref{eq:HenryRu}). As a result, we need to also keep track of the way the set $\operatorname{Aug}_1(S(K, \beta), \mathbb{F})$ changes when transitioning between these different diagrams for $S(K,\beta)$. This contributes to the factor appearing in front of~$\widetilde{R}_{S(K,\beta)}$ in the statement of Lemma~\ref{lem:count}.

\begin{Remark}
 In the case of $m$-graded representations with $m \neq 1$, $d=0$ is the only term in the disjoint union (\ref{eq:LeRu}) for grading reasons.
\end{Remark}

\subsection[Four diagrams for the satellite, $S(K, \beta)$]{Four diagrams for the satellite, $\boldsymbol{S(K, \beta)}$} \label{sec:satellite}

In establishing Lemma~\ref{lem:count} we will make use of four different (but Legendrian isotopic) versions of the satellite $S(K,\beta)$. The four $xy$-diagrams are denoted $S^1_{xy}(K,\beta)$, $S^2_{xy}(K, \beta)$, $S^1_{xz}(K,\beta)$, $S^2_{xz}(K,\beta)$ and will be defined momentarily; see Fig.~\ref{fig:SKbxy}. As a~preliminary, we consider $xz$-diag\-rams (front projection) and $xy$-diagrams (Lagrangian projection) for the companion \mbox{$K \subset J^1\R$} and pattern $\beta \subset J^1S^1$.

\subsubsection[Diagrams for $K$ and $\beta$]{Diagrams for $\boldsymbol{K}$ and $\boldsymbol{\beta}$}

For the companion knot $K \subset J^1\R$, apply Ng's resolution procedure (see~\cite{Ng1}) so that the $xy$-diagram for $K$ is related to the front projection of~$K$ by placing the strand with smaller slope on top at crossings and adding an extra loop at right cusps. Enumerate the Reeb chords of~$K$ by $a_1, \ldots, a_m$, and $c_1, \ldots, c_\ell$ where the $a_i$ correspond to crossings of the front projection of $K$ and the $c_i$ are the extra crossings near right cusps that arise from the resolution procedure. Choose an initial base point~$*$ of~$K$ not located on any of the loops at right cusps and at a point where the front diagram of~$K$ is oriented left-to-right. For convenience, we assume the $c_1, \ldots, c_\ell$ are enumerated in the order they appear when following along~$K$ according to its orientation, starting at~$*$.

For the braid $\beta \subset J^1S^1$, we form an $xy$-diagram as indicated in Fig.~\ref{fig:Dips} having $\ell$ dips in addition to the original crossings coming from the front projection of $\beta$. This is done by applying the resolution procedure to the front diagram of $\beta$ as in \cite[Section~2.2]{NgTraynor} (this amounts to adding a dip to the right of the crossings of $\beta$), and then adding $\ell-1$ extra dips to the $xy$-diagram. In addition, a collection of $n$ basepoints, $*_1, \ldots, *_n$, are placed immediately to the left of the crossings of $\beta$, one on each strand. Note that the addition of the dips can be accomplished by Legendrian isotopy as indicated in \cite[Section 3.1]{Sab}, and we have used a minor variation on the resolution procedure of \cite[Section 2.2]{NgTraynor} so that the dip arising from the resolution procedure has the same form as the others. Numbering the dips $1,\ldots, \ell$ as they appear from left to right in $T^*S^1$ (viewed as $[0,1]\times \R$ with boundaries identified) the $k$-th dip consists of two groups of crossings $x_{i,j}^k$ and $y_{i,j}^k$ for $1 \leq i < j \leq n$ that appear in the left and right half of the dip respectively. For $1 \leq k \leq \ell$, we collect these crossings as the entries of strictly upper triangular matrices denoted $X_k$ and $Y_k$. The crossings that correspond to $xz$-crossings of $\beta$ are labeled $p_1, \ldots, p_\lambda$, and DGA generators associated to the basepoints $*_1,\ldots, *_n$ will be labelled $s_1, \ldots, s_n$.

\begin{figure}
\labellist
\small
\pinlabel $z$ [l] at 4 26
\pinlabel $x$ [l] at 28 2
\pinlabel $y$ [l] at 250 26
\pinlabel $x$ [l] at 274 2
\pinlabel $X_1$ at 392 24
\pinlabel $Y_1$ at 452 24
\pinlabel $X_2$ at 542 24
\pinlabel $Y_2$ at 602 24
\pinlabel $\leadsto$ at 200 96
\pinlabel $*_3$ at 264 82
\pinlabel $*_2$ at 264 98
\pinlabel $*_1$ at 264 112

\endlabellist
\centerline{\includegraphics[scale=.7]{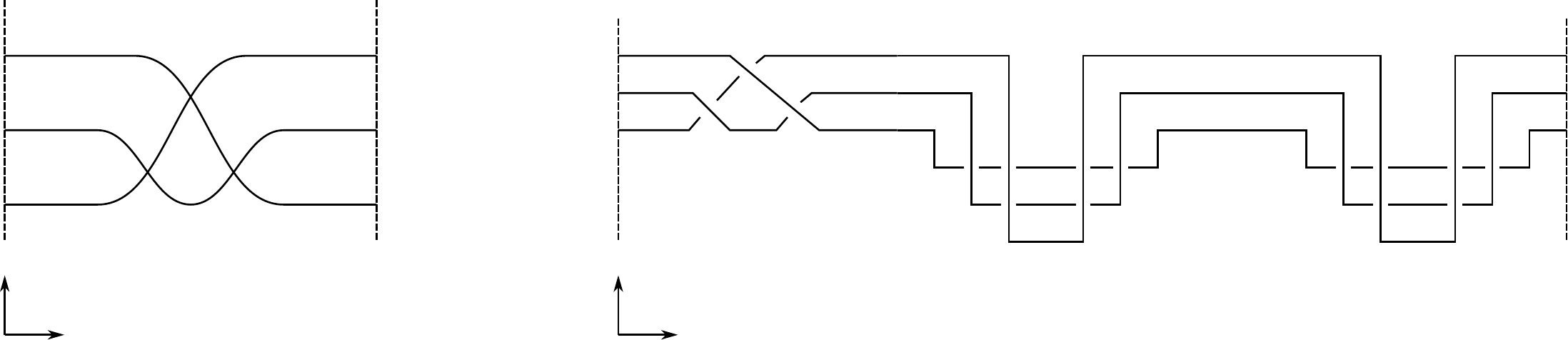}}

\caption{A Lagrangian diagram for the positive braid $\beta = \sigma_2 \sigma_1\sigma_2$ with $\ell =2$ dips.}
\label{fig:Dips}
\end{figure}
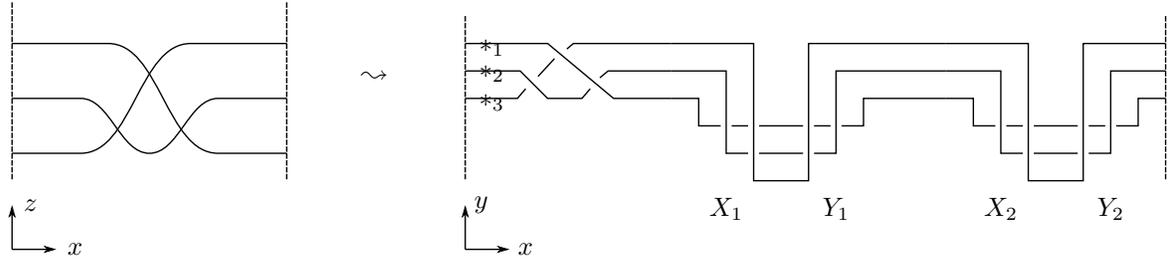

In constructing the diagrams $S^1_{xy}(K,\beta)$ and $S^2_{xy}(K,\beta)$ and comparing their DGAs, it is useful to be able to move the locations of dips around via a Legendrian isotopy. This may be accomplished by converting an isotopy $\phi_t\colon S^1 \rightarrow S^1$, $0\leq t \leq 1$, to an ambient contact isotopy
\begin{gather} \label{eq:contact}
\Phi_t\colon J^1S^1 \rightarrow J^1S^1,\qquad \Phi_t(x,y,z) = \left(\phi_t(x), \frac{y}{\phi'_t(x)}, z\right), \qquad 0\leq t \leq 1.
\end{gather}
In particular, given any cyclically ordered collection of points $x_1, \ldots, x_\ell \in S^1$, after a Legendrian isotopy the $X_k$ and $Y_k$ crossings can be arranged to appear above an arbitrarily small neighborhood of $x_k$ in $S^1$ for $1\leq k \leq \ell$.

\subsubsection{Diagrams for Legendrian satellites}

The Legendrian satellite $S(K,\beta) \subset J^1\R$ is formed by scaling the $y$ and $z$ coordinates of $J^1S^1$ so that $\beta$ sits in a small neighborhood, $N_0 \subset J^1S^1$, of the $0$-section and then applying a~contactomorphism $\Psi\colon N_0 \rightarrow N(K)$ of $N_0$ onto a Weinstein tubular neighborhood of~$K$. Two general methods for forming $xy$-diagrams of satellites are as follows:

\begin{itemize}\itemsep=0pt
\item {\it The $xy$-method.} Use the orientations of $K$ and $\R^2$ to identify an (immersed) annular neighborhood, $N_{xy}$, of the $xy$-diagram of $K$ with $S^1 \times (-\epsilon,\epsilon)$. Then, (after scaling the $y$-coordinate appropriately) place an $xy$-diagram for $\beta$ into $N_{xy}$. As $\Psi$ can be chosen to preserve the Reeb vector field (see \cite{Geiges}), this indeed produces an $xy$-diagram for $S(K,\beta)$.

\item {\it The $xz$-method.} First form an $xz$-diagram for $S(K,\beta)$ as in Section \ref{sec:sat}: start with the $n$-copy of $K$ (this is $n$-parallel copies of $K$ shifted a small amount in the $z$-direction), then insert the $xz$-projection of $\beta$
 at the location of the initial base point of $K$.
 Finally, produce an $xy$-diagram for $S(K,\beta)$ by applying Ng's resolution procedure.
\end{itemize}

{\bf Definition of $\boldsymbol{S^1_{xy}(K,\beta)}$:} To form $S^1_{xy}(K,\beta)$ expand $*$ to a cluster of base points $*_1, \ldots, *_\ell$ all appearing in order (according to the orientation of $K$) in a small neighborhood of~$*$. Then, form the satellite with the $xy$-diagram of $\beta$ described above in such a way that the crossings $p_1, \ldots, p_\lambda$ from $\beta$ and the crossings from $X_1$, $Y_1$ appear in a neighborhood of~$*_1$, and the crossings from $X_k, Y_k$ appear in a neighborhood of $*_k$ for $2 \leq k \leq \ell$.

{\bf Definition of $\boldsymbol{S^2_{xy}(K,\beta)}$:} This $xy$-diagram is formed from $S^1_{xy}(K,\beta)$ by using a Legendrian isotopy of $\beta$ (constructed from an isotopy of~$S^1$ as in~(\ref{eq:contact})) to relocate each collection of~$X_i$ crossings to a neighborhood of a right cusp of~$K$ (so that exactly one $X_i$ appears near each right cusp); relocate each group of $Y_i$ crossings to a neighborhood of a left cusp; and leave the crossings $p_1, \ldots, p_\lambda$ of $\beta$ (and the basepoints $*_1, \ldots, *_n$) in place at the location of the initial base point $*$ for~$K$. (Recall that~$\ell$ is both the number of dips and the number of right cusps of~$K$.)

{\bf Definition of $\boldsymbol{S^1_{xz}(K,\beta)}$ and $\boldsymbol{S^2_{xz}(K,\beta)}$:} Both $S^1_{xz}(K,\beta)$ and $S^2_{xz}(K,\beta)$ are formed using the $xz$-method. The only difference between the two is the placement of base points. For~$S^1_{xz}(K,\beta)$, base points $*_1, \ldots, *_n$ appear, one on each parallel strand of the $n$-copy, just before~$\beta$ (with respect to the orientation of $K$). For $S^2_{xz}(K,\beta)$, we have base points $*_1, \ldots, *_c$, one on each component of $S(K,\beta)$ placed on a loop near some right cusp of the component.

See Fig.~\ref{fig:SKbxy}.

\begin{figure}
\labellist
\small
\pinlabel $S^1_{xy}(K,\beta)$ [t] at 130 0
\pinlabel $S^2_{xy}(K,\beta)$ [t] at 474 0

\endlabellist
\centerline{\includegraphics[scale=.6]{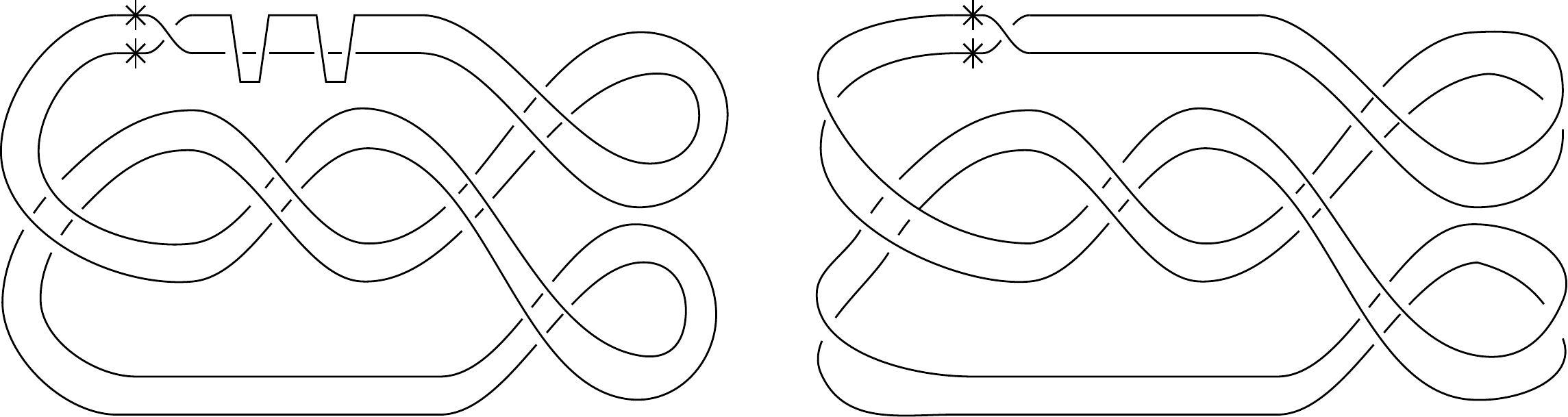}}

\quad

\quad

\labellist
\small
\pinlabel $S^1_{xz}(K,\beta)$ [t] at 126 0
\pinlabel $S^2_{xz}(K,\beta)$ [t] at 460 0
\endlabellist

\centerline{\includegraphics[scale=.6]{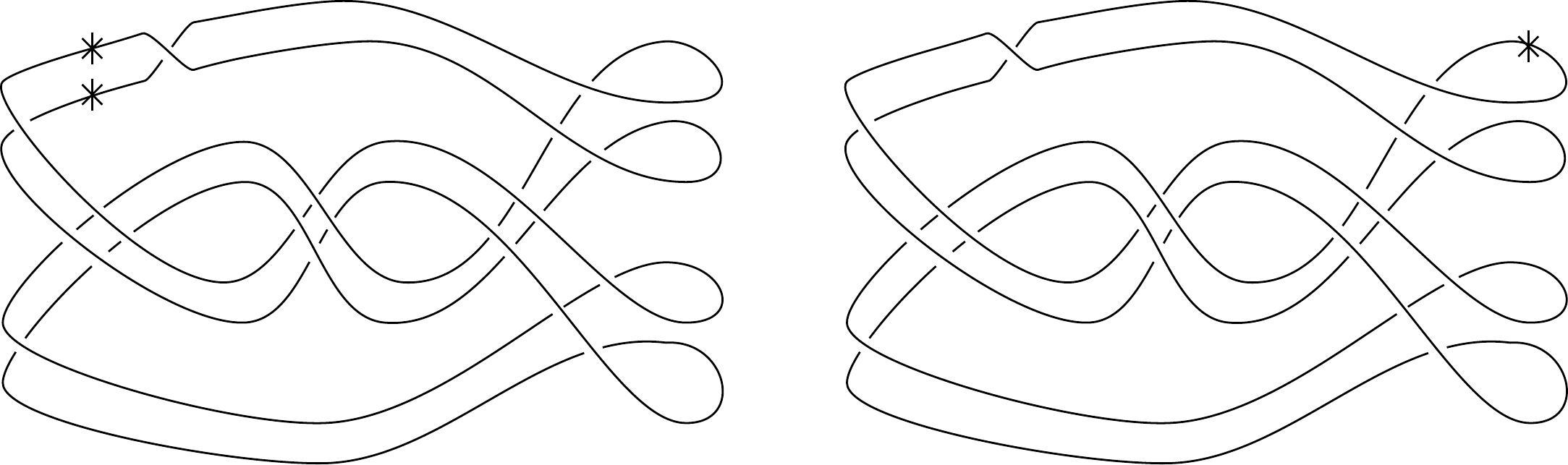}}

\vspace{2mm}

\caption{The Lagrangian ($xy$)-diagrams $S^1_{xy}(K,\beta)$, $S^2_{xy}(K,\beta)$, $S^1_{xz}(K,\beta)$, and $S^2_{xz}(K,\beta)$ where $K$ is a Legendrian trefoil and $\beta = \sigma_1 \in S_2$.}
\label{fig:SKbxy}
\end{figure}

\subsubsection{DGA generators} We set notations for the generators of the DGAs arising from the various $xy$-diagrams of $S(K,\beta)$ that have been defined. Many generators are indexed with a pair of subscripts $i$,~$j$. Such a~subscript indicates that at the overstrand of the crossing belongs to the $i$-th copy of $K$ and the understrand belongs to the $j$-th copy of $K$. Here,
 outside of an arc $A \subset K$ where the $p_1, \ldots, p_\lambda$ crossings of $\beta$ appear, $S(K,\beta)$ consists of $n$ copies of $K\setminus A$, which we label from $1$ to $n$ according to the descending order of their $y$-coordinates at the boundary of $A$.

{\bf Generators of $\boldsymbol{\alg(S^1_{xy}(K,\beta))}$ and $\boldsymbol{\alg(S^2_{xy}(K,\beta))}$:} The generating sets for these DGAs are in bijection. Both contain the DGA of $\beta$ as a sub-DGA, and this accounts for generators of the form $p_1, \ldots, p_\lambda$, $x_{i,j}^k$, $y^k_{i,j}$ for $1\leq k \leq \ell$ and $1\leq i < j \leq n$ as well as invertible generators $s_1, \ldots, s_n$ associated to the base points $*_1, \ldots, *_n$ on $\beta$. In addition, for each of the Reeb chords $a_1, \ldots, a_m$, and $c_1,\ldots, c_\ell$ of $K$ there are $n^2$ Reeb chords for $S^1_{xy}(K,\beta)$ and $S^2_{xy}(K,\beta)$ that we denote by $a^{k_1}_{i,j}$ and $c^{k_2}_{i,j}$, $1\leq k_1 \leq m$, $1\leq k_2 \leq \ell$, $1 \leq i,j \leq n$.

{\bf Generators of $\boldsymbol{\alg(S^1_{xz}(K,\beta))}$ and $\boldsymbol{\alg(S^2_{xz}(K,\beta))}$:} Both of these diagram have Reeb chords
\begin{itemize}\itemsep=0pt
\item $p_1, \ldots, p_\lambda$ from the $xz$-crossings of $\beta$;
\item $a^{k_1}_{i,j}$, $1 \leq k_1 \leq m$, $1 \leq i,j \leq n$, from the $xz$-crossings of~$K$;
\item $y^{k_2}_{i,j}$, $1 \leq k_2 \leq \ell$, $1 \leq i<j \leq n$, from crossings near left cusps;
\item $c^{k_2}_{i,j}$, $1 \leq k_2 \leq \ell$, $1 \leq j \leq i \leq n$, from crossings near right cusps.
\end{itemize}
The invertible generators coming from base points will be denoted as $s_1, \ldots, s_n$ for $S^1_{xz}(K,\beta)$ and as $r_1, \ldots, r_c$ for $S^2_{xz}(K,\beta)$.

In particular, note that for any of the four diagrams for $S(K,\beta)$ there is a collection of generators of the form $y^{k}_{i,j}$ which we will refer to as {\it $Y$-generators}.
\begin{Definition} Let $K'$ denote one of $S^1_{xy}(K,\beta)$, $S^2_{xy}(K,\beta)$, $S^1_{xz}(K,\beta)$, or $S^2_{xz}(K,\beta)$. We denote by
\[
\overline{\operatorname{Aug}}_1(K', \mathbb{F})_{Y=0}
\]
the set of augmentations of $\mathcal{A}(K')$ to $\mathbb{F}$ that map all $Y$-generators to~$0$.
\end{Definition}

The subsets $\overline{\operatorname{Aug}}_1(K', \mathbb{F})_{Y=0} \subset \overline{\operatorname{Aug}}_1(K', \mathbb{F})$ play an important role in the proof Lemma \ref{lem:count} as they correspond to representations with $d=0$ as well as to reduced normal rulings.

\begin{Remark}Computations of the differential for similar DGAs of $1$-dimensional Legendrian satellites are presented in detail in several places, and it is not difficult to extend these computations to give complete formulas for differentials in the present setting. See especially~\cite{LeRu} for satellites formed via the $xy$-method and~\cite{NgR2012} and~\cite{NRSSZ} for satellites formed via the $xz$-method. Rather than giving a complete description of differentials here, we state several partial formulas as they become useful in the following proofs.
\end{Remark}

\subsection{Representations and augmentations of the satellite} \label{sec:augsat}

We will use the following mild variation of \cite[Theorem 6.1]{LeRu} to transition between $n$-dimensional representations and augmentations of satellites.
As in the construction of $S^1_{xy}(K,\beta)$, expand the initial base point $*$ of $K$ into a cluster of base points $*_1, \ldots, *_\ell$ and consider the DGA $(\alg(K), \partial)$ with invertible generators $t_1, \ldots, t_\ell$.

\begin{Proposition} \label{prop:LeRu} Let $\beta$ be a reduced positive permutation braid, and construct $S^1_{xy}(K, \beta)$ as above.
There is a bijection
\[
\big\{ f \in \overline{\operatorname{Rep}}_1\big(K,\big(\mathbb{F}^n,0\big)\big) \,|\, f(t_1) \in B_\beta \mbox{ and } f(t_i) \in N_+ \mbox{ for $i \geq 2$}\big\} \quad \leftrightarrow \quad \overline{\aug}_1(S^1_{xy}(K, \beta), \mathbb{F})_{Y=0},
\]
where $B_\beta$ is the path subset of $\beta$ and $N_+ \subset {\rm GL}(n, \mathbb{F})$ is the group of upper triangular matrices with~$1$'s on the diagonal.
\end{Proposition}

\begin{proof}The proof is similar to \cite[Theorem~6.1]{LeRu} which implies the case of only~$1$ base point. We sketch the argument and highlight the modifications to the proof for the case of more than one basepoint. To avoid considering signs, we only treat the case where $\operatorname{char}(\mathbb{F}) = 2$ (which is the only case needed for Lemma~\ref{lem:count}). This allows us to work with the LCH DGA defined over~$\Z/2$ rather than over~$\Z$ (since any representation with $\operatorname{char}(\mathbb{F})=2$ factors through the change of coefficients map from the DGA over~$\Z$ to the DGA over~$\Z/2$).

 To relate the differential on $D\colon \alg\big(S^1_{xy}(K,\beta)\big) \rightarrow \alg\big(S^1_{xy}(K,\beta)\big)$ to the differential $\partial$ on $\alg(K)$ consider a $\Z/2$-algebra homomorphism
\[
\Phi\colon \ \alg(K) \rightarrow \operatorname{Mat}\big(n, \alg\big(S^1_{xy}(K,\beta)\big)\big)
\]
that sends Reeb chords $a_k$ or $c_k$ to the corresponding $n\times n$ matrices of Reeb chords $A_k = \big(a^k_{i,j}\big)$ or $C_k = \big(c^k_{i,j}\big)$ and satisfies
\begin{gather*} %\label{eq:Phi}
\Phi(t_1) = P^{xy}_{\beta},
\qquad \Phi(t_k) = (I+X_k), \qquad \mbox{for $2\leq k \leq \ell$},
\end{gather*}
where $P^{xy}_{\beta}$ and $P^{xz}_{\beta}$ denote the $xy$- and $xz$-path matrices of $\beta$ as defined in \cite[Section~4.1]{LeRu}. Over~$\Z/2$, with the differential $D$ on $\alg\big(S^1_{xy}(K,\beta)\big)$ extended entry-by-entry to \linebreak $\operatorname{Mat}\big(n, \alg\big(S^1_{xy}(K,\beta)\big)\big)$ we have the identities,
\begin{gather}
D Y_k = Y_k^2, \qquad 1 \leq k \leq \ell, \qquad \mbox{and} \label{eq:Y1} \\
D \circ \Phi(x) = \Phi \circ \partial(x) + O(Y), \qquad \mbox{for any generator $x \in\mathcal{A}(K)$}, \label{eq:Y2}
\end{gather}
where $O(Y)$ denotes a term belonging to the $2$-sided ideal generated by the $Y$-generators. This is essentially as in \cite[Section~5]{LeRu}; see especially Proposition~5.2 and Corollary~5.3 of~\cite{LeRu}. For generalizing to the case of more than one base point, note that, for $1 \leq k \leq \ell$, $\Phi(t_k)$ is the left-to-right $xy$-path matrix for the part of the $xy$-diagram of~$\beta$ that sits in a~neighborhood of the basepoint~$*_k$ on~$K$. As a result, the entries of $\Phi(t_k)$ (resp.\ $\Phi\big(t_k^{-1}\big)$)
 record the possibly negative punctures that boundaries of ``thick disks'' of~$S^1_{xy}(K,\beta)$ can have when they pass through the location of~$*_k$ on~$K$ in a way that agrees (resp.\ disagrees) with the orientation of~$K$; see \cite[Sections~4.1 and~5.2]{LeRu}.

Now, the bijection from the statement of the proposition arises from associating to an augmentation $\epsilon \in \overline{\aug}_1\big(S^1_{xy}(K, \beta), \mathbb{F}\big)_{Y=0}$ the matrix representation $f\colon  \mathcal{A}(K) \rightarrow \operatorname{Mat}(n,\mathbb{F})$ given by $f = \epsilon \circ \Phi$. From here~(\ref{eq:Y1}) and~(\ref{eq:Y2}) can then be used to show that under the assumption that $\epsilon(Y)=0$, the augmentation equation $\epsilon \circ D = 0$ is equivalent to the representation equation $f \circ \partial = 0$, cf.\ \cite[Theorem~6.1]{LeRu}. (Note that the hypothesis that $\beta$ is a reduced positive permutation braid is used as in~\cite{LeRu} to see that the equation $\epsilon(D \Phi(t_1)) =0 $ is equivalent to having $\epsilon \circ D(x) =0$ for all generators of the form $p_i$ or $x^1_{i,j}$.)
 \end{proof}

We are now prepared to give the proof of Lemma~\ref{lem:count}.

\begin{proof}[Proof of Lemma \ref{lem:count}]\quad

 {\bf Step 1.} Establish that
\begin{gather*}
\big|\overline{\operatorname{Rep}}_1\big(K, \big(\mathbb{F}_q^n,0\big), B_\beta\big)\big| = \frac{|{\rm GL}(n)|^{\ell-1}}{(q^{n(n-1)/2})^{\ell-1}} \cdot \big| \overline{\aug}_1\big(S^{1}_{xy}(K, \beta), \mathbb{F}_q\big)_{Y=0}\big|.
\end{gather*}

Keeping in mind that $\overline{\operatorname{Rep}}_1(K, (\mathbb{F}_q^n,0), B_\beta)$ is defined in terms of the DGA $\mathcal{A}(K, *_1, \ldots, *_\ell)$ of~$K$ equipped with $\ell$ base points $*_1, \ldots, *_\ell$, let $\overline{\operatorname{Rep}}_1\big(\mathcal{A}(K, *), \big(\mathbb{F}_q^n,0\big), B_\beta\big)$ denote the corresponding set of representations with respect to the DGA of~$K$ equipped with only the one initial base point~$*$, i.e., the set of ungraded representations $f\colon (\mathcal{A}(K,*), \partial) \rightarrow (\operatorname{Mat}(n, \mathbb{F}_q),0)$ ha\-ving \mbox{$f(t) \in B_\beta$}. The differential of a Reeb chord~$b$ of $K$ in $\mathcal{A}(K, *_1, \ldots, *_\ell)$ is obtained from its differential in~$\mathcal{A}(K,*)$ by replacing all occurrences of~$t$ by the product $t_1\cdots t_\ell$. Consequently,
\[
\big|\overline{\operatorname{Rep}}_1\big(K, \big(\mathbb{F}_q^n,0\big), B_\beta\big)\big| = |{\rm GL}(n)|^{\ell -1} \cdot \big|\overline{\operatorname{Rep}}_1\big(\mathcal{A}(K, *), \big(\mathbb{F}_q^n,0\big), B_\beta\big)\big|,
\]
where the factor $|{\rm GL}(n)|^{\ell-1}$ arises as the number of ways to factor a given matrix $f(t) \in B_\beta$ into a product $f(t_1) \cdots f(t_\ell)$ with $f(t_i) \in {\rm GL}(n)$ for $1 \leq i \leq \ell$. On the other hand, using Proposition~\ref{prop:LeRu} gives
\[
\big| \overline{\aug}_1(S^{1}_{xy}(K, \beta), \mathbb{F}_q)_{Y=0}\big| = \big(q^{n(n-1)/2}\big)^{\ell-1} \big|\overline{\operatorname{Rep}}_1\big(\mathcal{A}(K, *), \big(\mathbb{F}_q^n,0\big), B_\beta\big)\big|,
\]
where in this case the factor $\big(q^{n(n-1)/2}\big)^{\ell-1}$ arises as the number of ways to factor a given matrix $f(t) \in B_\beta$ into a product $f(t_1) \cdots f(t_\ell)$ with $f(t_1) \in B_\beta$ and $f(t_i) \in N_+$ for $2 \leq i \leq \ell$. Here, we use the connection with the Bruhat decomposition from \cite[Section~4.3]{LeRu} so that~$B_\beta$ has the form $BS_{\beta} B$ where $S_\beta$ is a permutation matrix and~$B$ is the group of invertible upper triangular matrices. Therefore, given $f(t) \in B_\beta$ and strictly upper-triangular matrices \mbox{$f(t_2), \ldots, f(t_{\ell}) \in N_+$} there is a unique element $f(t_1) \in B_\beta$ so that $f(t) = f(t_1) f(t_2)\cdots f(t_{\ell})$.

{\bf Step 2.} Establish that
\begin{gather*}
\big| \overline{\aug}_1\big(S^{1}_{xy}(K, \beta), \mathbb{F}_q\big)_{Y=0}\big| = \big| \overline{\aug}_1\big(S^{2}_{xy}(K, \beta), \mathbb{F}_q\big)_{Y=0}\big|.
\end{gather*}

A Legendrian isotopy of $\beta$ as in (\ref{eq:contact}) can be used to produce a Legendrian isotopy of $S(K,\beta)$ that moves the $X_k$ and $Y_k$ crossings around the annular neighborhood $N_{xy}$ of the Lagrangian projections of~$K$. In particular, this procedure leads to a Legendrian isotopy $\Lambda_t$, $0 \leq t \leq 1$, from $S^1_{xy}(K,\beta)$ to $S^2_{xy}(K,\beta)$. Note that the crossings $p_1, \ldots, p_\lambda$ and the base points $*_1, \ldots, *_n$ can be assumed to remain in place during the isotopy. Moreover, since the $xy$-diagram of $\beta$ has no vertical tangencies, by scaling the $y$- and $z$- coordinates by an appropriately small factor, it can be assumed that, for all $0\leq t \leq 1$, the $xy$-diagram of $\Lambda_t$ does not have self-tangencies. As a result, the Reeb chords of $\Lambda_t$ appear in continuous $1$-parameter families parametrized by $t \in [0,1]$, and (identifying the corresponding generators of all $\mathcal{A}(\Lambda_t)$)
 the differential remains constant except for a finite number of handleslide disk bifurcations.
 These occur when an~$X_k$ or~$Y_k$ crossing of $\beta$ passes over or under another strand of the satellite as in Move~I of \cite[Section~6]{ENS}. During the move, there are three Reeb chords, one that is a crossing of $\beta$ and two more of the form $a^k_{i,j}$ or $c^k_{i,j}$, that all come together at a triple point. Label the three Reeb chords as $x$, $y$, $z$ so that their lengths (i.e., the difference of $z$-coordinates at endpoints) satisfy $h(x) > h(y) >h(z)$, and note that $z$ is the crossing from $\beta$. The DGAs before and after the triple point move are related by a DGA isomorphism $\phi\colon (\mathcal{A}, \partial) \rightarrow (\mathcal{A}', \partial')$ that maps~$x$ to an element of the form $x \pm yz$ or $x\pm zy$ and fixes all other generators. In particular, $\phi$ restricts to the identity on the sub-algebra generated by the $Y$-generators.
Composing all of the DGA isomorphisms from handleslide disks, we see that there is a DGA isomorphism $\varphi\colon \big(\mathcal{A}\big(S^1_{xy}(K, \beta)\big), \partial_1\big) \rightarrow \big(\mathcal{A}\big(S^2_{xy}(K, \beta)\big), \partial_2\big)$ that restricts to the identity on all $Y$-generators. As a result, $\varphi^*\colon \overline{\aug}_1\big(S^{2}_{xy}(K, \beta), \mathbb{F}_q\big) \rightarrow \overline{\aug}_1\big(S^{1}_{xy}(K, \beta), \mathbb{F}_q\big)$, $\varphi^*\epsilon = \epsilon \circ \varphi$ induces the required bijection between $\overline{\aug}_1\big(S^{1}_{xy}(K, \beta), \mathbb{F}_q\big)_{Y=0}$ and $\overline{\aug}_1\big(S^{2}_{xy}(K, \beta), \mathbb{F}_q\big)_{Y=0}$.

{\bf Step 3.} Establish that
 \begin{gather} \label{eq:Step3}
\big| \overline{\aug}_1\big(S^{2}_{xy}(K, \beta), \mathbb{F}_q\big)_{Y=0}\big| = \big(q^{n(n-1)/2}\big)^\ell\big| \overline{\aug}_1\big(S^{1}_{xz}(K, \beta), \mathbb{F}_q\big)_{Y=0}\big|.
\end{gather}

The generators of $S^1_{xz}(K,\beta)$ are identified with a subset of the generators of $S^2_{xy}(K,\beta)$, and this leads to an algebra inclusion $i\colon \mathcal{A}\big(S^1_{xz}(K,\beta)\big) \rightarrow \mathcal{A}\big(S^2_{xy}(K,\beta)\big)$. In fact, it is not hard to check that~$i$ is a DGA homomorphism; see \cite[Proposition~4.23]{NRSSZ} for a detailed explanation in the case where $\beta$ is the identity braid. The difference between the two DGAs is that $\mathcal{A}\big(S^2_{xy}(K,\beta)\big)$ has additional generators of the form $c^k_{i,j}$ and $x^k_{i,j}$ with $1 \leq k \leq \ell$, $1 \leq i<j \leq n$ that $\mathcal{A}\big(S^1_{xz}(K,\beta)\big)$ does not have. Equation~(\ref{eq:Step3}) then follows from:

{\bf Claim:} Given any $\epsilon \in \overline{\aug}_1\big(S^{1}_{xz}(K, \beta), \mathbb{F}_q\big)_{Y=0}$ and arbitrary values $\epsilon'\big(c^k_{i,j}\big) \in \mathbb{F}_q$, there exists a unique augmentation $\epsilon' \in \overline{\aug}_1\big(S^{2}_{xy}(K, \beta), \mathbb{F}_q\big)_{Y=0}$ extending these values and restricting to $\epsilon$ on $\mathcal{A}\big(S^{1}_{xz}(K, \beta)\big)$.

Given $\epsilon$ and $\epsilon'\big(c^k_{i,j}\big)$ we need to show that there are unique values $\epsilon'\big(x^k_{i,j}\big)$ for which the equations
\[
\epsilon' \circ \partial\big(c^k_{i,j}\big) = 0 \qquad \mbox{and} \qquad \epsilon' \circ \partial\big(x^k_{i,j}\big) =0, \qquad 1 \leq k \leq \ell, \  1 \leq i<j\leq n
\]
hold. Since $\partial x^k_{i,j}$ belongs to the $2$-sided ideal generated by the $Y$-generators, and $\epsilon$ vanishes on all $Y$-generators, the equations $\epsilon' \circ \partial\big(x^k_{i,j}\big) =0$ are satisfied. Note that (again computing over~$\Z/2$)
\[
\partial C_k = (I+X_k)^{\pm 1} + W_k,
\]
where $W_k$ denotes a matrix with entries in the subalgebra generated by $\mathcal{A}_{xz}^1(S(K,\beta))$ and by the~$c^k_{i,j}$ generators. The map $\mathbb{F}_q^{n(n-1)/2} \rightarrow \mathbb{F}_q^{n(n-1)/2}$ that sends a collection of values $\big(\epsilon'\big(x^k_{i,j}\big)\big)_{i<j}$ to the above diagonal part of $\epsilon'\big((I+X_k)^{\pm 1}\big)$ is a bijection. Hence, there is a unique way to choose $\epsilon'\big(x^k_{i,j}\big)$ so that the equations $\epsilon \circ\partial c^k_{i,j}=0$, $1 \leq i<j\neq n$ hold (since this is equivalent to having $\epsilon'\big((I+X_k)^{\pm1}\big) = \epsilon'(W_k)$ hold on the upper triangular entries.)

{\bf Step 4.} Establish that
 \begin{gather} \label{eq:Step4}
\big| \overline{\aug}_1\big(S^{1}_{xz}(K, \beta), \mathbb{F}_q\big)_{Y=0}\big| = (q-1)^{n-c}\big| \overline{\aug}_1\big(S^{2}_{xz}(K, \beta), \mathbb{F}_q\big)_{Y=0}\big|.
\end{gather}

First, note that when the locations of basepoints are moved around the number of $Y=0$ augmentations does not change. (This is because when a basepoint~$s_i$ moves through an $xy$-crossing~$a$, the DGAs before and afterward are related by an isomorphism $\phi$ that maps~$a$ to an element of the form $s_i^{\pm1} a$ or $a s_i^{\pm1}$ (depending on orientation of $K$ and whether $s_i$ passes the upper or lower endpoint of~$a$) and fixes all other generators; see \cite[Theorem~2.20]{NgR2012}. In particular, the induced bijection between augmentation sets, $\phi^*$, preserves the $Y=0$ subsets (since $\epsilon(a) = 0$ if and only if $\epsilon\big(t_i^{\pm1}a\big) =0$).

To understand the effect of changing the number of basepoints, suppose that $K_1$ and $K_2$ are two $xy$-diagrams, identical except that on $K_1$ a collection of base points with corresponding generators $u_1, \ldots, u_s$ appears near the location of a single base point $u$ of $K_2$. Arguing as in Step 1, shows that there is a DGA map $\phi\colon \mathcal{A}(K_2) \rightarrow \mathcal{A}(K_1)$ with $\phi(u) = u_1\cdots u_s$ and $\phi(x) = x$ for all other generators, and moreover $\phi^*\colon \overline{\aug}_1(K_1, \mathbb{F}_q) \rightarrow \overline{\aug}_1(K_2, \mathbb{F}_q)$ is surjective and $(q-1)^{s-1}$-to-$1$. (Here, $(q-1)^{s-1}$ represents the number of ways to factor an element of~$\mathbb{F}_q^*$ into a product of~$s$ elements in $\mathbb{F}_q^*$.) Since $\phi$ is the identity on Reeb chords, when applied to~$S(K,\beta)$, $\phi^*$~restricts to a~surjective, $(q-1)^{s-1}$-to-$1$ map between the $Y=0$ augmentation sets. Starting with $S^1_{xz}(K,\beta)$ and applying this procedure repeatedly with the collection of base points on each component (keeping in mind that $S^1_{xz}(K,\beta)$ has $n$ base points while $S^2_{xz}(K,\beta)$ has $c$) leads to the formula~(\ref{eq:Step4}).

{\bf Step 5.} Establish a decomposition of the form
\[
 \overline{\aug}_1\big(S^{2}_{xz}(K, \beta), \mathbb{F}_q\big)_{Y=0} = \bigsqcup_\rho W_\rho,
\]
where the disjoint union is over reduced normal rulings, $\rho$, and
\[
|W_\rho| = (q-1)^{j(\rho)+c} \cdot q^{\frac{1}{2}(-j(\rho) +{\rm rb}(S(K,\beta)))}
\]
with $c$ the number of components of $S(K,\beta)$ and ${\rm rb}(S(K,\beta))$ the number of Reeb chords of~$S_{xz}^2(K,\beta)$.

A decomposition of the entire augmentation variety $\overline{\aug}_1\big(S^{2}_{xz}(K, \beta), \mathbb{F}_q\big)$ (without imposing the $Y=0$ condition) as $\bigsqcup W_\rho$ with the disjoint union over all normal rulings (without the reduced condition) is established in Theorem~3.4 of~\cite{HenryRu}. The statement of that theorem implies that $|W_\rho| = (q-1)^{j(\rho)-c}q^{b(\rho)}$ where $b(\rho)$ is the number of ``returns''\footnote{Given a normal ruling $\rho$ for a Legendrian link $K'$, the $xz$-crossings of $K'$ that are not switches are either {\it departures} or {\it returns}. At a departure (resp.\ return) the normality condition holds to the left (resp.\ right) of the crossing but not to the right (resp.\ left) of the crossing. See \cite[Section~3]{NgSab}.} of $\rho$ plus the number of right cusps of $\rho$, and Lemma~5 of~\cite{NgSab} shows that $b(\rho) = \frac{1}{2} ( -j(\rho) + {\rm rb}(S(K,\beta)))$. Thus, it suffices to show that an augmentation $\epsilon \in \overline{\aug}_1\big(S^{2}_{xz}(K, \beta), \mathbb{F}_q\big)$ belongs to $W_\rho$ with $\rho$ {\it reduced} if and only if the $Y=0$ condition is satisfied.

A summary of the construction of the decomposition $\overline{\aug}_1\big(S^{2}_{xz}(K, \beta), \mathbb{F}_q\big)=\bigsqcup W_\rho$ is as follows. Let $K' \subset J^1\R$ be a Legendrian in plat position whose DGA is computed from resolving the front projection of $K'$ and positioning a single base point on each component of $K'$ in the loop near some chosen right cusp. The DGA of $S^2_{xz}(K,\beta)$ is of this required type. For such a~$K'$, the article~\cite{HenryRu} considers objects called Morse complex sequences (MCSs) that consist of a sequence of chain complexes and formal handleslide marks (which are vertical segments on the front diagram of $K'$) subject to several axioms motivated by Morse theory. Section~5 of \cite{HenryRu} gives a bijection between $\overline{\aug}_1(K', \mathbb{F}_q)$ and the set of ``$A$-form''\footnote{An MCS is in {\it $A$-form} if its handleslides only appear in specified locations on the front diagram of $K'$ to the left of crossings and right cusps. See \cite[Section~5]{HenryRu}.} MCSs for $K'$, denoted here $\operatorname{MCS}^A(K')$, where an augmentation $\epsilon$ corresponds to an $A$-form MCS with one handleslide mark just to the left of every crossing or right cusp, $x$, with $\epsilon(x) \neq 0$ with the handleslide coefficient determined by the value $\epsilon(x)$. In Section~4.1 of~\cite{HenryRu}, another class of MCSs called ``$SR$-form'' MCSs are considered; we will denote the set of all $SR$-form MCSs for~$K'$ as~$\operatorname{MCS}^{SR}(K')$. Each $SR$-form MCS, $\mathcal{C}$, has an associated normal ruling~$\rho$ of~$K'$, and all handleslide marks of $\mathcal{C}$ appear in collections of a~standard form near switches, returns, and right cusps of~$\rho$. In particular, at every switch of~$\rho$,~$\mathcal{C}$ must have a collection of handleslides with non-zero coefficients; returns and right cusps may or may not have handleslides. In Section~6 of~\cite{HenryRu} a bijection $\Psi\colon \operatorname{MCS}^A(K') \rightarrow \operatorname{MCS}^{SR}(K')$ and its inverse $\Phi=\Psi^{-1}$ are constructed. By definition, $\epsilon \in W_\rho$ if the corresponding $A$-form MCS, $\mathcal{C}_\epsilon$, is such that the $SR$-form MCS $\Psi(\mathcal{C}_\epsilon)$ has associated normal ruling~$\rho$.

Note that for an $A$-form or $SR$-form MCS every handleslide has some associated $xz$-crossing or right cusp. (For an $SR$-form MCS with normal ruling $\rho$, associated crossings can only be switches or returns of $\rho$ and a single crossing may have more than one handleslide associated to it.) Let $\operatorname{MCS}^A(S(K,\beta))_{Y=0}$ and $\operatorname{MCS}^{SR}(S(K,\beta))_{Y=0}$ denote those $A$-form and $SR$-form MCSs for $S^2_{xz}(K,\beta)$ that do not have any handleslide marks associated to $xz$-crossings corresponding to the $Y$-generators. (These are the groups of $n(n-1)/2$ crossings that appear in $S^2_{xz}(K,\beta)$ near the location of left cusps of~$K$.) Step~5 is then completed by the following.

\begin{Lemma}The above constructions from {\rm \cite{HenryRu}} restrict to bijections
\[
 \overline{\aug}_1\big(S^{2}_{xz}(K, \beta), \mathbb{F}_q\big)_{Y=0} \leftrightarrow \operatorname{MCS}^A(S(K,\beta))_{Y=0} \leftrightarrow \operatorname{MCS}^{SR}(S(K,\beta))_{Y=0}
\]
and an $SR$-form MCS $\mathcal{C}$ belongs to $\operatorname{MCS}^{SR}(S(K,\beta))_{Y=0}$ if and only if the associated normal ruling $\rho$ is reduced.
\end{Lemma}
\begin{proof}The first bijection is clear since an $A$-form MCS has no handleslides at $Y$ crossings if and only if the corresponding augmentation vanishes on $Y$-generators. Turning to the second bijection, given an $A$-form or $SR$-form MCS, $\mathcal{C}$, let $i(\mathcal{C})$ denote the first (from left to right) $xz$-crossing of $S^2_{xy}(K,\beta)$ that has a handleslide associated to it. The defining requirement for both $\operatorname{MCS}^A(S(K,\beta))_{Y=0}$ and $\operatorname{MCS}^{SR}(S(K,\beta))_{Y=0}$ is equivalent to $i(\mathcal{C})$ not being one of the $Y$-crossings. (The $Y$-crossings all appear to the left of the other $xz$-crossings of $S^2_{xz}(K,\beta)$ since $K$ is in plat position.) Examining the definition of $\Psi$ and $\Phi$ in Section 6 of \cite{HenryRu}, it is straightforward to see that $i(\Psi(\mathcal{C})) = i(\mathcal{C})$ and $i(\Phi(\mathcal{C})) = i(\mathcal{C})$, so that $\Psi$ and $\Phi$ restrict to provide the bijection $\operatorname{MCS}^A(S(K,\beta))_{Y=0} \leftrightarrow \operatorname{MCS}^{SR}(S(K,\beta))_{Y=0}$.

For the final statement of the lemma, note that a normal ruling of $S^2_{xz}(K,\beta)$ is reduced if and only if it has no switches at $Y$-crossings. See \cite[Lemma~3.2]{NgR2012}. Moreover, if a $Y$-crossing is a~return then there must be another $Y$-crossing somewhere to its left that is a switch. (If there are no switches in the $Y$-crossings then, it is easy to see that all $Y$-crossings are departures.) Thus, for $\mathcal{C}$ in $SR$-form, $i(\mathcal{C})$ is a $Y$-crossing if and only if the corresponding ruling is not reduced.
\end{proof}

{\bf Step 6.} Completion of the proof.

Combining the identities from Steps 1--5, we have
\begin{align*}
\big|\overline{\operatorname{Rep}}_1\big(K, \big(\mathbb{F}_q^n,0\big), B_\beta\big)\big| & = |{\rm GL}(n)|^{\ell-1}  q^{n(n-1)/2} (q-1)^{n-c}  \sum_{\rho} (q-1)^{j(\rho)+c}q^{\frac{1}{2}(-j(\rho) + {\rm rb}(S(K,\beta)))} \\
& = |{\rm GL}(n)|^{\ell-1}  q^{n(n-1)/2} (q-1)^{n}   q^{\frac{1}{2}{\rm rb}(S(K,\beta))}  \sum_{\rho} (q^{1/2}-q^{-1/2})^{j(\rho)} \\
& = |{\rm GL}(n)|^{\ell-1}  q^{n(n-1)/2} (q-1)^{n}   q^{n^2 {\rm rb}(K)/2} q^{\lambda(\beta)/2} \widetilde{R}_{S(K,\beta)}(z),
\end{align*}
where the summations are over all {\it reduced} normal rulings and at the last equality we used that the number of Reeb chords of $S^2_{xz}(K,\beta)$ is $n^2 \cdot {\rm rb}(K)+\lambda(\beta)$.
\end{proof}

\section{The multi-component case} \label{sec:multicomp}

For simplicity, we have restricted the focus of this article to the case where $K$ is a connected Legendrian knot. We close with a discussion of an appropriate modification of Theorem~\ref{thm:main} for the case of Legendrian links with multiple components.

When $K = \sqcup_{i =1}^c K_i$ is a Legendrian link with $c$ components, for $\vec{n} = (n_1, \ldots, n_c)$ with $n_i\geq 1$ one can consider the $\vec{n}$-colored Kauffman polynomial defined by satelliting each $K_i$ with the symmetrizer from the $n_i$-stranded BMW algebra in a multi-linear manner. With the $\vec{n}$-colored $1$-graded ruling polynomial defined by the analogous modification,
\[
R^1_{\vec{n},K}(q) = \left(\prod_{i=1}^c\frac{1}{c_{n_i}} \right)\sum_{\vec{\beta} \in S_{n_1} \times \cdots \times S_{n_c}} \left(\prod_{i=1}^c q^{\lambda(\beta_i)/2} \right) \widetilde{R}_{S(K,\vec{\beta})}(z)\vert_{z = q^{1/2}-q^{-1/2}},
\]
the second equality of Theorem \ref{thm:main}, modified to read $R^1_{\vec{n},K}(z)=F_{\vec{n},K}(a,q)|_{a^{-1}=0}$, follows by a~mild variation on the arguments of Sections~\ref{sec:ncolored} and~\ref{sec:Kauffman}.

To obtain an additional equality with a total representation number, one should work with the so-called {\it composable algebra} version of the LCH DGA, $(\mathcal{A}_{{\rm comp}}, \partial)$, cf.~\cite{BEE, EENS}. The underlying algebra $\mathcal{A}_{{\rm comp}}$ has generators $b_1, \ldots, b_r$ and $t_1^{\pm1}, \ldots, t_{\ell}^{\pm1}$ from Reeb chords and basepoints as well as idempotent generators $e_1, \ldots, e_c$ corresponding to the components of $K$. Moreover, $\mathcal{A}_{{\rm comp}}$ has relations
\begin{alignat*}{3}
& e_ie_j = \delta_{i,j}, \qquad & & \displaystyle \sum_{i=1}^c e_i = 1,& \\
& e_i b_k = \delta_{i,u(k)} b_k, & & b_ke_i = \delta_{i, l(k)} b_k, & \\
 & e_i t_k = t_k e_i = \delta_{i, s(k)} t_k, \qquad & & t_kt_k^{-1} = t_k^{-1}t_k = e_{s(k)}, &
\end{alignat*}
where the upper (resp.\ lower) endpoint of the Reeb chord~$b_k$ is on the component $K_{u(k)}$ (resp.\ $K_{l(k)}$) and the basepoint $t_k$ sits on the~$K_{s(k)}$ component. Note that an algebra representation $f\colon \mathcal{A}_{{\rm comp}} \rightarrow \operatorname{End}(V)$ is equivalent to a~collection of vector spaces $V_1, \ldots, V_c$ together with linear maps
\[
f(b_k)\colon \ V_{l(k)} \rightarrow V_{u(k)}
\]
assigned to Reeb chords, and invertible linear maps
\[
f(t_k)\colon \ V_{s(k)} \rightarrow V_{s(k)}
\]
assigned to base points. Here, the $V_i$ are determined from $V$ via $V_i = f(e_{i}) V$. (This is as in the correspondence between quiver representations and representations of the corresponding path algebra, see, e.g., \cite[Section~1.2]{Brion}. Except for the relation $t_kt_k^{-1} = t_k^{-1}t_k = e_{s(k)}$, the composable algebra $\mathcal{A}_{{\rm comp}}$ is precisely the path algebra associated to the quiver with vertices indexed by components of $K$ and edges corresponding to Reeb chords and base points.) The DGA differential $\partial$ on $\mathcal{A}_{{\rm comp}}$ is defined to satisfy $\partial e_i = \partial t_k= 0$ and by the usual holomorphic disk count on Reeb chords.

For $\vec{n}= (n_1, \ldots, n_c)$, with $n_i \geq 1$ as above, we denote by $\overline{\operatorname{Rep}}_1\big(K, \big(\mathbb{F}^{\vec{n}}_q, 0\big)\big)$
the set of DGA representations $f\colon (\mathcal{A}_{\rm comp}, \partial) \rightarrow \big(\operatorname{End}\big(\oplus_{i} \mathbb{F}_q^{n_i}\big), 0\big)$ that,
when viewed as quiver representations, assign the collection of vector spaces $\mathbb{F}_q^{n_1}, \ldots, \mathbb{F}_q^{n_c}$ to the components of~$K$, i.e., $f(e_i)$ is the projection to the $\mathbb{F}_{q}^{n_i}$ component. Note that to obtain a Legendrian isotopy invariant we need to adjust the normalizing factor from~\cite{LeRu} used in Definition~\ref{def:total}. This is done by defining the ($1$-graded) {\it total $\vec{n}$-dimensional representation number} of~$K$ to be
\begin{align*}
\operatorname{Rep}_1\big(K, \mathbb{F}^{\vec{n}}_q\big) & := \left(\prod_{i,j} \big|\operatorname{Hom}_{\mathbb{F}_q}\big(\mathbb{F}^{n_j}_q, \mathbb{F}^{n_i}_q\big)\big|^{-{\rm rb}_{i,j}(K)/2}\right)
\\
 & \quad {}\times \left( \prod_i \big|{\rm GL}\big(n, \mathbb{F}^{n_i}_q\big)\big|^{-\ell_i} \right)  \big| \overline{\operatorname{Rep}}_1\big(K, \big(\mathbb{F}^{\vec{n}}_q, 0\big)\big)\big| \\
 & = \left(\prod_{i,j} \big(q^{n_in_j}\big)^{-{\rm rb}_{i,j}(K)/2}\right) \\
 & \quad {} \times \left(\prod_{i} \left(q^{n_i(n_i-1)/2}   \prod_{m=1}^{n_i}\big(q^m-1\big)\right)^{-\ell_i} \right)  \big| \overline{\operatorname{Rep}}_1\big(K, \big(\mathbb{F}^{\vec{n}}_q, 0\big)\big) \big|,
\end{align*}
 where ${\rm rb}_{i,j}(K)$ is the number of Reeb chords, $b_k$, with $u(k) =i$ and $l(k) = j$ and $\ell_i$ is the number of basepoints on the $K_i$ component of $K$.

With the composable algebra used for $K$, a suitable modification of Theorem~6.1 from~\cite{LeRu} relating augmentations of the satellites $S\big(K, \vec{\beta}\big)$ with higher dimensional representations of $(\mathcal{A}_{{\rm comp}}(K), \partial)$ continues to hold. [The key point being that when components $K_i$ and $K_j$ are satellited with~$n_i$ and~$n_j$ stranded braids respectively, each Reeb chord, $b_k$, with $u(k) = i$ and $l(k) =j$ corresponds to an $n_i \times n_j$ matrix of Reeb chords in the satellite $S(K, \vec{\beta})$.] From this starting point, the analog of the first equality of Theorem~\ref{thm:main}, $\operatorname{Rep}_1\big(K, \mathbb{F}^{\vec{n}}_q\big)= R^1_{\vec{n},K}(z)$, can be deduced as in Section~\ref{sec:5}.

\subsection*{Acknowledgements}
This article is dedicated to Dmitry Fuchs to whom the second author is grateful for his generosity and support through years in grad school and beyond. Thank you Dmitry! DR acknowledges support from Simons Foundation grant~\#429536.

\pdfbookmark[1]{References}{ref}
\LastPageEnding

\end{document}